\documentclass{article} 

\usepackage{amsmath,amsthm,amssymb,amsfonts}

\usepackage{geometry} 
\newtheorem{Theorem}{Theorem}[section]
\theoremstyle{plain}

\newtheorem{Lemma}[Theorem]{Lemma}
\newtheorem{Definition}[Theorem]{Definition}
\newtheorem{Remark}[Theorem]{Remark}
\newtheorem{Corollary}[Theorem]{Corollary}
\newtheorem{Proposition}[Theorem]{Proposition}

\numberwithin{equation}{section}

\title{Branching Brownian Motion with spatially-homogeneous and point-catalytic branching}
\author{Sergey Bocharov\footnote{S.Bocharov:  
Department of Mathematics, Zhejiang University, Zheda Road, Hangzhou 310027, China, 
e-mail: bocharov@zju.edu.cn. The author is supported by NSFC grant (No.11731012)}, 
Li Wang\footnote{Li Wang: 
School of Sciences, Beijing University of Chemical Technology, Beijing, P.R. China, 
e-mail: wangli@mail.buct.edu.cn. The author is supported by NSFC grant (No.
11301020)}}
\begin{document}
\maketitle
\begin{abstract}
We consider a model of Branching Brownian Motion in which the usual spatially-homogeneous and catalytic branching at a single point are simultaneously present. We establish the almost sure growth rates of population in certain time-dependent regions and as a consequence the first-order asymptotic behaviour of the rightmost particle.
\end{abstract}
\section{Introduction and main results}
\subsection{Description of the model}
We are going to consider a model of Branching Brownian Motion in which the usual spatially-homogeneous and catalytic branching at a single point are present simultaneously. 

Such a process is initiated with a single particle at time $0$ whose spatial position at time $t \geq 0$ up until the time it dies is given by $X_t$ where $(X_t)_{t \geq 0}$ is distributed like a standard Brownian motion. We let $T'$ and $T_0$ be two random times which are independent conditional on $(X_t)_{t \geq 0}$ and satisfy $P(T' > t | (X_s)_{s \geq 0}) = \mathrm{e}^{- \beta t}$ 
and $P(T_0 > t | (X_s)_{s \geq 0}) = \mathrm{e}^{- \beta_0 L_t}$, where $\beta > 0$ and $\beta_0 > 0$ are some constants and $(L_t)_{t \geq 0}$ is the local time at $0$ of $(X_t)_{t \geq 0}$. Note that almost surely $X_{T_0} = 0$ and $X_{T'} \neq 0$.

At the time $T' \wedge T_0$ the initial particle dies and is replaced with a random number of new particles. If $T_0 < T'$ then the number of new particles follows 
some given distribution $(q_n)_{n \geq 1}$. Otherwise it follows a different distribution $(p_n)_{n \geq 1}$.

All the new particles then independently of each other and of the past stochastically repeat the behaviour of their parent starting from $X_{T' \wedge T_0}$. That is,
they move like Brownian motions, die after random times giving births to new particles, etc.

Note that all particles always produce at least one child upon their death ruling out the possibility of population extinction.

An equivalent (up to indistinguishability) description would be to say that after a random time $T$ such that $P(T > t | (X_s)_{s \geq 0}) = \mathrm{e}^{- \beta_0 L_t - \beta t}$ the initial particle dies, and 
at position $x$ where it died it is replaced with a random number $A(x)$ of new particles where for $n \geq 1$
\begin{equation}
\label{Ax}
P(A(x) = n) = \left\{
\begin{array}{rl}
q_n& \text{if } x = 0, \\
p_n& \text{if } x \neq 0 
\end{array} \right.
\end{equation}
and the new particles then stochastically repeat the behaviour of their parent starting from $x$.
Thus the model can be thought of as the BBM model with spatially-inhomogeneous branching rate $\beta_0 \delta_0(\cdot) + \beta$, where $\delta_0(\cdot)$ is the Dirac delta function, and spatially-inhomogeneous offspring distribution given by \eqref{Ax} since informally we may say that $L_t = \int_0^t \delta_0(X_s) \mathrm{d}s$ (this can be made formal via the theory of additive functionals of Brownian motion).

Our model combines in a natural way the classical BBM model with constant branching and the BBM model with a single catalytic point. The first one has been studied for many decades and numerous asymptotic results are available for it (see e.g. \cite{M75}, \cite{B78}, \cite{LS87}, \cite{R13}). The catalytic model has been given less attention (one may for example look at \cite{K99} for a general review of the topic).

\subsection{Some notation}
Using common practice we label the initial particle by $\varnothing$ and all the other particles according to the Ullam-Harris convention. So that for example particle ``$\varnothing32$" is the second child of the third child of the initial particle.   

For two particles $u$ and $v$ we shall write $u < v$ if $u$ is an ancestor of $v$ So for example 
$\varnothing < \varnothing3 < \varnothing32$. We shall write $|u|$ for the number of ancestors of the particle $u$. So for example $|\varnothing32| = 2$.

We denote the set of all particles in the system at time $t$ by $N_t$ and for every particle $u \in N_t$ we let $X^u_t$ denote its spatial position at time $t$ and $(X^u_s)_{s \in [0,t]}$ its histrorical path up to time $t$ with $L_t^u$ the local time at $0$ of $(X^u_s)_{s \in [0,t]}$. We also define
\[ 
R_t := \sup \{X^u_t \ : \ u \in N_t\}
\] 
to be the position of the rightmost particle at time $t$.

We let $m_0 = \sum_{n \geq 1}n q_n$ be the mean of the offspring distribution due to catalytic branching and $m = \sum_{n \geq 1}n p_n$ the mean of 
the offspring distribution due to homogeneous branching. For convenience we also define effective branching rates
\[
\hat{\beta} := \beta(m-1)
\] 
and
\[
\hat{\beta}_0 := \beta_0(m_0 - 1) \text{.}
\]
Finally, we let $(\mathcal{F}_t)_{t \geq 0}$ be the natural filtration of the branching process and $P$ 
the associated probability measure with the corresponding expectation $E$.

\subsection{Motivation}
In this subsection we present a few simple calculations which should motivate our main results in the next subsection.
 
Let us define for any $x \in \mathbb{R}$ and $t \geq 0$ 
\begin{equation}
\label{set_ntx}
N_t^x := \{ u \in N_t \ : \ X^u_t >x\}
\end{equation}
to be the set of particles in the system at time $t$ whose spatial position is bigger than $x$. A simple application of the widely used  "many-to-one" formula (see Subsection 2.2) gives an exact expression 
for $E|N_t^x|$:
\begin{Proposition}
\label{prop_enx}
For any $x \geq 0$ and $t \geq 0$
\begin{equation}
\label{eq_enx}
E |N_t^x| = \Phi \Big( \hat{\beta}_0 \sqrt{t} - \frac{x}{\sqrt{t}}\Big) \exp \Big\{ \frac{1}{2} \hat{\beta}_0^2t -\hat{\beta}_0x + \hat{\beta}t \Big\} \text{,}
\end{equation}
where $\Phi(x) = \mathbb{P} (\mathcal{N}(0, 1) \leq x) = \int_{- \infty}^x (2 \pi)^{-1/2} \mathrm{e}^{- y^2/2} \mathrm{d}y$ is the cumulative 
distribution function of a standard normal random variable.
\end{Proposition}
In particular, for any $\lambda \geq 0$ we get 
\begin{equation}
\label{enlt}
E|N_t^{\lambda t}|  = \Phi \Big( \big(\hat{\beta}_0 - \lambda\big) \sqrt{t}\Big) \exp \Big\{ \big(\frac{1}{2}\hat{\beta}_0^2 - \hat{\beta}_0\lambda + \hat{\beta}\big)t\Big\}.
\end{equation}
Using the fact that $\Phi(x) \sim (2 \pi)^{-1/2}|x|^{-1} \mathrm{e}^{-x^2/2}$ as $x \to - \infty$ and $\Phi(x) \to 1$ as $x \to \infty$ we can then see that 
\begin{equation}
\label{log_e_lambda}
\frac{1}{t} \log E|N_t^{\lambda t}| \to \Delta_\lambda \text{ as } t \to \infty \text{,}
\end{equation}
where
\begin{equation}
\label{delta_lambda}
\Delta_{\lambda} = \left\{
\begin{array}{rl}
\frac{1}{2} \hat{\beta}_0^2 - \hat{\beta}_0 \lambda + \hat{\beta}& \text{if } \lambda \leq \hat{\beta}_0, \\
- \frac{1}{2} \lambda^2 + \hat{\beta}& \text{if } \lambda \geq \hat{\beta}_0 .
\end{array} \right. 
\end{equation}
We can then observe that $\Delta_\lambda$ takes positive or negative values according to whether $\lambda < \lambda_{crit}$ 
or $\lambda > \lambda_{crit}$, where 

\begin{equation}
\label{lambda_crit}
\lambda_{crit} = \left\{
\begin{array}{rl}
\frac{\hat{\beta}}{\hat{\beta}_0} + \frac{1}{2} \hat{\beta}_0& \text{if } \hat{\beta} \leq \frac{1}{2}\hat{\beta}_0^2, \\
\sqrt{2 \hat{\beta}}& \text{if } \hat{\beta} \geq \frac{1}{2}\hat{\beta}_0^2 .
\end{array} \right.
\end{equation}
Since the expected number of particles above the line $(\lambda_{crit} + \epsilon)t$ decays exponentially with $t$ and the expected number of particles 
above the line $(\lambda_{crit} - \epsilon)t$ grows exponentially with $t$ we may interpret $\lambda_{crit}$ for now as the speed of the rightmost particle 
"in expectation".

Also using symmetry or a direct calculation we may find the expected total population at any time $t \geq 0$:
\begin{equation}
\label{eq_en}
E |N_t| = 2 E |N_t^0| = 2 \Phi \Big( \hat{\beta}_0 \sqrt{t}\Big) \exp \Big\{ \big( \frac{1}{2} \hat{\beta}_0^2 + \hat{\beta}\big) t \Big\} \text{.}
\end{equation}

In particular,
\begin{equation}
\label{exp_pop_assymp}
E |N_t| \sim 2  \exp \Big\{ \big( \frac{1}{2} \hat{\beta}_0^2 + \hat{\beta}\big) t \Big\} \text{ as } t \to \infty \text{.}
\end{equation}

\subsection{Main results}
Our aim is to replace convergences in expectation from the previous subsection with the almost sure convergences.

For all the results in this subsection we shall have to impose an additional condition on the offspring distribution commonly known as the $X \log X$ condition (for some discussion see e.g. \cite{K04}, \cite{Lyons} or \cite{LYS13}):
\begin{equation}
\label{xlogx}
\sum_{n \geq 1} p_n n \log n < \infty \text{ and } \sum_{n \geq 1} q_n n \log n < \infty \text{.}
\end{equation}
This condition is needed to ensure that certain martingales have non-zero limits as we shall see in Subsection 2.3.

Our first result, which should be compared with \eqref{exp_pop_assymp}, is the almost sure approximation of the population size.
\begin{Lemma}
\label{population_as}
Suppose that condition \eqref{xlogx} on the offspring distribution is satisfied. Then 
\begin{equation}
\label{eq_pop_as}
\lim_{t \to \infty} \frac{1}{t} \log |N_t| = \frac{1}{2} \hat{\beta}_0^2 + \hat{\beta} \qquad P \text{-a.s.}
\end{equation}
\end{Lemma}
Next, unarguably the most important result of this paper, is the almost sure approximation of 
$|N_t^{\lambda t}|$ and it should be compared with \eqref{log_e_lambda}.
\begin{Lemma}
\label{pop_lambda_as}
Suppose that condition \eqref{xlogx} is satisfied. Take any $\lambda > 0$ and let $\Delta_\lambda$ be as in \eqref{delta_lambda} and $\lambda_{crit}$ as in 
\eqref{lambda_crit}.

If $\lambda < \lambda_{crit}$ then 
\begin{equation}
\label{subcrit}
\lim_{t \to \infty} \frac{1}{t} \log |N_t^{\lambda t}| = \Delta_\lambda \ (> 0) \qquad P \text{-a.s.}
\end{equation}

If $\lambda > \lambda_{crit}$ then
\begin{equation}
\label{supercrit1}
\lim_{t \to \infty} |N_t^{\lambda t}| = 0 \qquad P \text{-a.s.}
\end{equation}

and furthermore
\begin{equation}
\label{supercrit2}
\lim_{t \to \infty} \frac{1}{t} \log P (|N_t^{\lambda t}| > 0) = \Delta_{\lambda} \ (< 0) \text{.}
\end{equation}
\end{Lemma}
As a direct corollary of Lemma \ref{pop_lambda_as} we establish the almost sure speed of the rightmost particle.
\begin{Corollary}
\label{cor_rightmost}
Suppose that condition \eqref{xlogx} is satisfied. Then 
\begin{equation}
\label{rightmost}
\lim_{t \to \infty} \frac{R_t}{t} = \lambda_{crit} \qquad P \text{-a.s.,}
\end{equation}
where $\lambda_{crit}$ is given in \eqref{lambda_crit}.
\end{Corollary}

\subsection{Outline of the paper}
The rest of this article is organised as follows. In Subsection 2.1 we follow the standard procedure of extending the probability space by constructing the spine process over our branching system. Then in Subsection 2.2 we recall the Many-to-One formula for branching processes and apply it to prove Proposition \ref{prop_enx} as well as the upper bounds for Lemma \ref{population_as} and Lemma \ref{pop_lambda_as}. In Subsection 2.3 we construct certain change-of-mesure martingales and use them to prove lower bound for Lemma \ref{population_as} as well as ``preliminary" lower bounds for Lemma \ref{pop_lambda_as}. In subsections 3.1 and 3.2 we give the heuristic argument and the formal proof for lower bounds in Lemma \ref{pop_lambda_as}. Finally, we conclude the paper with the proof of Corollary \ref{cor_rightmost}.

\section{Spine results and applications}

\subsection{Spine construction}
In this section we extend our probability space by introducing the spine process. A more detailed description of this procedure may be found for example in \cite{HH09}.

The spine of the branching process, which we shall denote by $\xi$, is the infinite line of descent of particles chosen uniformly at random from all possible lines of descent. It is constructed in the following way. The initial particle $\varnothing$ of the branching process begins the spine. When the initial particle dies and is replaced with a random number of new particles, one of them is chosen uniformly at random to continue the spine. This procedure is then repeated recursively: whenever the particle in the spine dies, one of its children is chosen uniformly at random to continue the spine. We may then write the spine as 
$\xi = \{ \xi_0, \xi_1, \xi_2, \cdots\}$, where $\xi_n$ is the lable of the spine particle in the $n$th generation and $\xi_0 = \varnothing$. 

We let $\tilde{P}$ denote the probabiity measure under which the branching process is defined together with the spine. Hence $P = \tilde{P}|_{\mathcal{F}_\infty}$. We let $\tilde{E}$ be the expectation corresponding to $\tilde{P}$.

Below introduce some new notation in relation with the spine process.

For any $t \geq 0$ we let $node_t(\xi) := u$ for the unique particle $u \in N_t \cap \xi$. That is, 
$node_t(\xi)$ is the label of the spine particle at time $t$.

For any $t \geq 0$ we let $\xi_t := X^u_t$ for the unique particle $u \in N_t \cap \xi$. So that $\xi_t$ is the spatial position of the spine particle at time $t$. It is not hard to check that the process $(\xi_t)_{t \geq 0}$ is a Brownian motion under $\tilde{P}$. We shall denote by $(\tilde{L}_t)_{t \geq 0}$ its local time at the origin.

For any $t \geq 0$ we let $n_t := n$ for the unique $n$ such that $\xi_n \in N_t$. So that $(n_t)_{t \geq 0}$ is the counting process of the number of fissions that have occurred along the path of the spine by time $t$. We denote the sequence of such fission times by $S_n$ and the number of particles produced at each such fission by $A_n$, $n \geq 1$.

Moreover, we would like to distinguish fissions along the spine that occurred due to catalytic branching from those that occurred due to homogeneous branching. In order to do so we denote the branching times along the spine that took place when the spine was at the origin by $S_n^0$ and the number of particles produced at these times  by $A_n^0$, $n \geq 1$. Similarly, we denote the branching times along the spine when it was not at the origin by $S_n'$and the number of particles produced at these times by $A_n'$, $n \geq 1$. We also denote the counting processes for $(S_n^0)_{n \geq 1}$ and $(S_n')_{n \geq 0}$ by $(n_t^0)_{t \geq 0}$ and $(n_t')_{t \geq 0}$ respectively.

Observe that conditional on the path of the spine $(\xi_t)_{t \geq 0}$, 
$(n_t^0)_{t \geq 0}$ and $(n_t')_{t \geq 0}$ are independent (inhomogeneous in the first case) Poisson processes (or Cox processes) with jump rates $\beta_0 \delta_0(\cdot)$ and $\beta$ respectively so that
\[
\tilde{P}\Big( n_t^0 = k \big\vert (\xi_s)_{0 \leq s \leq t}\Big) = \frac{(\beta_0 \tilde{L}_t)^k}{k!} \mathrm{e}^{- \beta_0 \tilde{L}_t}.
\]
and
\[
\tilde{P}\Big( n_t' = k \big\vert (\xi_s)_{0 \leq s \leq t}\Big) = \frac{(\beta t)^k}{k!} \mathrm{e}^{- \beta t}
\]
Finally, it is convinient to define several filtrations of the now extended probability space in order to take various conditional expectations.
\begin{Definition}[Filtrations]
\label{filtrations} $ $
\begin{itemize}
\item $\mathcal{F}_t$ was defined in Subsection 1.2. It is the filtration which contains all the information about all the particles' motion and their genealogy. It doesn't however have any information about the spine.
\item $\tilde{\mathcal{F}}_t := \sigma \big(\mathcal{F}_t, \ node_t(\xi)\big)$. Thus $\tilde{\mathcal{F}}$ has all the information about the branching process and all the information about the spine. This will be the largest filtration.
\item $\mathcal{G}_t := \sigma \big( \xi_s : 0 \leq s \leq t \big)$. This filtration only contains information about the path of the spine but it doesn't know which particles make up the spine along its path at different times.
\item $\tilde{\mathcal{G}}_t := \sigma \big( \mathcal{G}_t, \ node_s(\xi) : 0 \leq s \leq t , \ A^0_n : n \leq n^0_t, \ A_m' : m \leq n_t'\big)$. This filtration has information about both the path of the spine, its genealogy and how many particles are born along the path of the spine. However it has no information  about anything happenning off the spine.
\end{itemize}
\end{Definition}
We note that $\mathcal{F}_t \subset \tilde{\mathcal{F}}_t$ and $\mathcal{G}_t \subset \tilde{\mathcal{G}}_t \subset \tilde{\mathcal{F}}_t$.

\subsection{Many-to-One lemma and applications}
The proof of the following result with a detailed discussion can be found for example in \cite{HH09} or \cite{HR}.
\begin{Lemma}[Many-to-One Lemma]
\label{many_to_one_lemma}
Let $Y$ be a non-negative $\tilde{\mathcal{F}}_t$-measurable random variable. It can be decomposed as
\[
Y = \sum_{u \in N_t} Y(u) \mathbf{1}_{\{ node_t(\xi) = u \}} \text{,}
\]
where for all $u \in N_t$, $Y(u)$ is $\mathcal{F}_t$-measurable and then 
\[
E^x \Big( \sum_{u \in N_t} Y(u) \Big) = \tilde{E}^x \Big( Y \mathrm{e}^{\hat{\beta}_0 \tilde{L}_t + \hat{\beta} t}\Big) \text{,}
\]
In particular, if $f$ is a sufficiently nice functional then
\begin{equation}
\label{many_to_one_eq}
E^x \Big[ \sum_{u \in N_t} f \Big( \big(X^u_s\big)_{s \in [0,t]}\Big) \Big] = \tilde{E}^x \Big[ f\Big(\big(\xi_s\big)_{s \in [0,t]} \Big) 
\mathrm{e}^{\hat{\beta}_0 \tilde{L}_t + \hat{\beta} t}\Big] \text{.}
\end{equation}
\end{Lemma}
Let us now apply \eqref{many_to_one_eq} to prove equations \eqref{eq_enx} and \eqref{eq_en} given as 
the motivation in the first section.
\begin{proof}[Proof of Proposition \ref{prop_enx} and identity \eqref{eq_en}]
Take $x \geq 0$ and $t \geq 0$. Then
\begin{align}
\label{eq_enx1}
E |N_t^x| &=  E \sum_{u \in N_t} \mathbf{1}_{\{X^u_t > x\}}\nonumber\\
&= \tilde{E} \Big[ \mathbf{1}_{\{\xi_t > x\}} \mathrm{e}^{\hat{\beta}_0 \tilde{L_t} + \hat{\beta} t}\Big] \text{.}
\end{align}
We now make use of the joint density of $\xi_t$ and $\tilde{L}_t$ (which for example can be found in \cite{KS84}):
\[
\tilde{P} \Big( \xi_t \in \mathrm{d}y, \ \tilde{L}_t \in \mathrm{d} l\Big) = 
\frac{|y| + l}{\sqrt{2 \pi t^3}} \exp \Big\{ - \frac{(|y| + l)^2}{2t}\Big\} \mathrm{d}y \mathrm{d}l \qquad y \in \mathbb{R}, \ l \geq 0 
\]
to complete the proof of \eqref{eq_enx}.
\begin{align*}
\tilde{E} \Big[ \mathbf{1}_{\{\xi_t > x\}} \mathrm{e}^{\hat{\beta_0} \tilde{L_t} + \hat{\beta} t}\Big] &= 
\mathrm{e}^{\hat{\beta} t} \int_0^\infty \int_x^{\infty} \mathrm{e}^{\hat{\beta}_0 l} \frac{y + l}{\sqrt{2 \pi t^3}} 
\exp \Big\{ - \frac{(y+l)^2}{2t}\Big\} \mathrm{d}y \mathrm{d}l\\
&= \mathrm{e}^{\hat{\beta} t} \int_0^\infty \mathrm{e}^{\hat{\beta}_0 l} \frac{1}{\sqrt{2 \pi t}} 
\exp \Big\{ - \frac{(x+l)^2}{2t}\Big\} \mathrm{d}l\\
&= \mathrm{e}^{\hat{\beta} t} \int_0^\infty \frac{1}{\sqrt{2 \pi t}} \exp \Big\{ - \frac{1}{2t} \big( l - (\hat{\beta}_0t - x)\big)^2 
+ \frac{\hat{\beta}_0^2}{2}t - \hat{\beta}_0 x \Big\} \mathrm{d}l\\
&= \mathrm{e}^{\hat{\beta}t + \frac{\hat{\beta}_0^2}{2}t - \hat{\beta}_0 x} \int_{- (\hat{\beta}_0 \sqrt{t} - \frac{x}{\sqrt{t}})}^\infty 
\frac{1}{\sqrt{2 \pi}} \mathrm{e}^{- \frac{z^2}{2}} \mathrm{d}z\\
&= \Phi \Big( \hat{\beta}_0 \sqrt{t} - \frac{x}{\sqrt{t}}\Big) \exp \Big\{ \frac{1}{2} \hat{\beta}_0^2t -\hat{\beta}_0x + \hat{\beta}t \Big\} \text{.}
\end{align*}
For the expected total population we could have followed a similar calculation:
\begin{equation}
\label{eq_en2}
E|N_t| = \tilde{E} \Big[ \mathrm{e}^{\hat{\beta_0} \tilde{L_t} + \hat{\beta} t} \Big] = \cdots = 
2 \Phi \Big( \hat{\beta}_0 \sqrt{t}\Big) \exp \Big\{ \big( \frac{1}{2} \hat{\beta}_0^2 + \hat{\beta}\big) t \Big\}
\end{equation}

\end{proof}
Let us now prove the upper bound for Lemma \ref{population_as}.
\begin{Proposition}[Upper bound for Lemma \ref{population_as}]
\label{upper_pop}
\begin{equation}
\label{eq_upper_pop}
\limsup_{t \to \infty} \frac{1}{t} \log |N_t| \leq \frac{1}{2} \hat{\beta}_0^2 + \hat{\beta} \qquad P \text{-a.s.}
\end{equation}
\end{Proposition}
\begin{proof}
Fix $\epsilon > 0$. Then by the Markov inequality and \eqref{eq_en} 
\begin{align*}
P \Big( \frac{1}{n} \log |N_n| > \frac{1}{2}\hat{\beta}_0^2 + \hat{\beta} + \epsilon \Big) &= 
P \Big( |N_n| > \mathrm{e}^{(\frac{1}{2}\hat{\beta}_0^2 + \hat{\beta} + \epsilon)n} \Big)\\ 
&\leq \mathrm{e}^{-(\frac{1}{2}\hat{\beta}_0^2 + \hat{\beta} + \epsilon)n} E|N_n|
< 2 \mathrm{e}^{- \epsilon n}.
\end{align*}
It follows from the Borel-Cantelli Lemma that 
\[
P \Big( \Big\{ \frac{1}{n} \log |N_n| > \frac{1}{2}\hat{\beta}_0^2 + \hat{\beta} + \epsilon \Big\} \text{ i.o. }\Big) = 0
\]
and thus 
\[
\limsup_{n \to \infty} \frac{1}{n} \log |N_n| \leq \frac{1}{2} \hat{\beta}_0^2 + \hat{\beta} + 
\epsilon \qquad P \text{-a.s.}
\]
By letting $\epsilon \to 0$ we establish \eqref{eq_upper_pop} with the limit taken over integer times. To 
get convergence over any real-valued sequence we note that $(|N_t|)_{t \geq 0}$ is a non-decreasing process and so for any $t > 0$
\[
\frac{1}{t}\log |N_t| \leq \frac{\lceil t \rceil}{t} \frac{\log |N_{\lceil t \rceil}|}{\lceil t \rceil}
\]
and hence
\[
\limsup_{t \to \infty} \frac{1}{t} \log |N_t| \leq \limsup_{t \to \infty} 
\frac{\log |N_{\lceil t \rceil}|}{\lceil t \rceil}
\leq \frac{1}{2} \hat{\beta}_0^2 + \hat{\beta} \qquad P \text{-a.s.}
\]
\end{proof}
For upper bounds of Lemma \ref{pop_lambda_as} we need to adjust the previous argument because 
unlike $(|N_t|)_{t \geq 0}$ the process $(|N_t^{\lambda t}|)_{t \geq 0}$ is not monotone. We first  establish the following result.
\begin{Proposition}
\label{lambda_sup}
For $\lambda > 0$ and $n \in \mathbb{N}\cup\{0\}$ we define the following set of particles:
\[
\hat{N}^{\lambda n}_n := \big\{ u \in N_{n+1} \ : \ \sup_{s \in [n, n+1]}X^u_s \geq \lambda n\big\}.
\]
Then 
\[
\limsup_{n \to \infty} \frac{1}{n} \log E |\hat{N}^{\lambda n}_n| \leq \Delta_\lambda.
\]
\end{Proposition}
Note that for any $t \in [n, n+1]$ it is always true that 
$|N_t^{\lambda t}| \leq |\hat{N}^{\lambda n}_n|$.
\begin{proof}
By the Many-to-One Lemma we have
\begin{align*}
E |\hat{N}^{\lambda n}_n| &= E \sum_{u \in N_{n+1}} 
\mathbf{1}_{\{\sup_{s \in [n, n+1]}X^u_s \geq \lambda n\}}\\
&= \tilde{E} \Big[ \mathbf{1}_{\{\sup_{s \in [n, n+1]}\xi_s \geq \lambda n\}} 
\mathrm{e}^{\hat{\beta}_0 \tilde{L}_{n+1} + \hat{\beta}(n+1)}\Big]\\
&= \tilde{E} \Big[ \mathbf{1}_{\{\xi_{n+1} + \bar{\xi}_n \geq \lambda n\}} 
\mathrm{e}^{\hat{\beta}_0 \tilde{L}_{n+1} + \hat{\beta}(n+1)}\Big],
\end{align*}
where $\bar{\xi}_n := \sup_{s \in [n, n+1]} (\xi_s - \xi_{n + 1})$ and 
$\bar{\xi}_n \stackrel{d}{=} \sup_{s \in [0,1]} \xi_s \stackrel{d}{=} |\mathcal{N}(0,1)|$
under $\tilde{P}$.

Then for any $\delta \in (0, \lambda)$ we can split the latter expectation as 
\begin{align*}
 \tilde{E} \Big[ \mathbf{1}_{\{\xi_{n+1} + \bar{\xi}_n \geq \lambda n\}} 
\mathrm{e}^{\hat{\beta}_0 \tilde{L}_{n+1} + \hat{\beta}(n+1)}\Big] &= 
 \tilde{E} \Big[ \mathbf{1}_{\{\xi_{n+1} + \bar{\xi}_n \geq \lambda n\}} 
\mathrm{e}^{\hat{\beta}_0 \tilde{L}_{n+1} + \hat{\beta}(n+1)} 
\mathbf{1}_{\{|\xi_{n+1}| \leq (\lambda - \delta)n\}}
\Big]\\
&+ \tilde{E} \Big[ \mathbf{1}_{\{\xi_{n+1} + \bar{\xi}_n \geq \lambda n\}} 
\mathrm{e}^{\hat{\beta}_0 \tilde{L}_{n+1} + \hat{\beta}(n+1)} 
\mathbf{1}_{\{|\xi_{n+1}| > (\lambda - \delta)n\}}
\Big]
\end{align*}
We shall refer to the first term in the sum as $I_1$ and the second one as $I_2$. First we show that 
the contribution of $I_1$ is negligibly small as it has a faster than exponential decay rate.
\begin{align*}
I_1 &= \tilde{E} \Big[ \mathbf{1}_{\{\xi_{n+1} + \bar{\xi}_n \geq \lambda n\}} 
\mathrm{e}^{\hat{\beta}_0 \tilde{L}_{n+1} + \hat{\beta}(n+1)} 
\mathbf{1}_{\{|\xi_{n+1}| \leq (\lambda - \delta)n\}}\Big]\\
&\leq \tilde{E} \Big[ \mathrm{e}^{\hat{\beta}_0 \tilde{L}_{n+1} + \hat{\beta}(n+1)} 
\mathbf{1}_{\{\bar{\xi}_n \geq \delta n\}} \Big]\\
&\leq \Big( \tilde{E} \Big[ \mathrm{e}^{2\hat{\beta}_0 \tilde{L}_{n+1} + 2\hat{\beta}(n+1)} \Big]
\Big)^{\frac{1}{2}} \Big( \tilde{P} \big( \bar{\xi}_n \geq \delta n \big) \Big)^{\frac{1}{2}}
\end{align*}
using Cauchy-Schwarz inequality in the last line.Then as we know from \eqref{eq_en2}
\[
\frac{1}{n} \log
\Big( \tilde{E} \Big[ \mathrm{e}^{2\hat{\beta}_0 \tilde{L}_{n+1} + 2\hat{\beta}(n+1)} \Big]
\Big)^{\frac{1}{2}} \to \hat{\beta}_0^2 + \hat{\beta} \qquad \text{ as }n \to \infty
\]
while, since $\bar{\xi}_n \stackrel{d}{=} |\mathcal{N}(0,1)|$,
\[
\frac{1}{n^2} \log \Big( \tilde{P} \big( \bar{\xi}_n \geq \delta n \big) \Big)^{\frac{1}{2}} \to 
\frac{-\delta^2}{4} \qquad \text{ as } n \to \infty.
\]
Thus 
\begin{equation}
\label{i1}
\limsup_{n \to \infty} \frac{1}{n^2} \log I_1 \leq - \frac{\delta^2}{4}.
\end{equation}
On the other hand,
\begin{align*}
I_2 &= \tilde{E} \Big[ \mathbf{1}_{\{\xi_{n+1} + \bar{\xi}_n \geq \lambda n\}} 
\mathrm{e}^{\hat{\beta}_0 \tilde{L}_{n+1} + \hat{\beta}(n+1)} 
\mathbf{1}_{\{|\xi_{n+1}| > (\lambda - \delta)n\}}
\Big]\\
&\leq \tilde{E} \Big[ \mathrm{e}^{\hat{\beta}_0 \tilde{L}_{n+1} + \hat{\beta}(n+1)} 
\mathbf{1}_{\{|\xi_{n+1}| > (\lambda - \delta)n\}}\Big]\\
&= 2\tilde{E} \Big[ \mathrm{e}^{\hat{\beta}_0 \tilde{L}_{n+1} + \hat{\beta}(n+1)} 
\mathbf{1}_{\{\xi_{n+1} > (\lambda - \delta)n\}}\Big]\\
&= 2E|N_{n+1}^{(\lambda - \delta)n}|
\end{align*} 
using symmetry in the third line and identity \eqref{eq_enx1} in the fourth line. Thus from 
\eqref{eq_enx} we can see (just as we did in \eqref{enlt} - \eqref{delta_lambda}) that
\begin{equation}
\label{i2}
\limsup_{n \to \infty} \frac{1}{n} \log I_2 \leq \Delta_{\lambda - \delta}.
\end{equation}
From \eqref{i1} and \eqref{i2} we have that for any $\delta \in (0, \lambda)$
\[
\limsup_{n \to \infty} \frac{1}{n} \log E |\hat{N}^{\lambda n}_n| 
= \limsup_{n \to \infty} \frac{1}{n} \log(I_1 + I_2) \leq \Delta_{(\lambda-\delta)}.
\]
Letting $\delta \to 0$ and using continuity and monotonicity of $\Delta_{\lambda}$ as a function of $\lambda$ we obtain the sought result.
\end{proof}

Proposition \ref{lambda_sup} can now be applied to prove the upper bounds for Lemma \ref{pop_lambda_as}.
\begin{Proposition}[Upper bounds for Lemma \ref{pop_lambda_as}]
\label{upper_pop_lambda}$ $

If $\lambda < \lambda_{crit}$ ($\Delta_\lambda > 0$) then 
\begin{equation}
\label{subcrit_upper}
\limsup_{t \to \infty} \frac{1}{t} \log |N_t^{\lambda t}| \leq \Delta_\lambda 
\qquad P \text{-a.s.}
\end{equation}

If $\lambda > \lambda_{crit}$ ($\Delta_\lambda < 0$) then
\begin{equation}
\label{supercrit1_upper}
\lim_{t \to \infty} |N_t^{\lambda t}| = 0 \qquad P \text{-a.s.}
\end{equation}
and
\begin{equation}
\label{supercrit2_upper}
\limsup_{t \to \infty} \frac{1}{t} \log P (|N_t^{\lambda t}| > 0) \leq \Delta_{\lambda}  \text{.}
\end{equation}
\end{Proposition}
\begin{proof}
For any $\lambda >0$ let $\hat{N}_n^{\lambda n}$ be as in the previous proposition and fix $\epsilon > 0$. Then the Markov inequality gives 
\[
P \Big( |\hat{N}_n^{\lambda n}| > \mathrm{e}^{(\Delta_\lambda + \epsilon)n} \Big) \leq 
\mathrm{e}^{-(\Delta_\lambda + \epsilon)n} E |\hat{N}_n^{\lambda n}|
\]
and from Proposition \ref{lambda_sup} the right hand side decays exponentially fast. Therefore by Borel-Cantelli Lemma
\[
P \Big( \Big\{|\hat{N}_n^{\lambda n}| > \mathrm{e}^{(\Delta_\lambda + \epsilon)n} \Big\} \text{ i.o.}\Big) = 0.
\]
This is equivalent to saying that
\[
|\hat{N}_n^{\lambda n}| \leq \mathrm{e}^{(\Delta_\lambda + \epsilon)n} \text{ eventually } P \text{-a.s.}
\]
So $P$-almost surely for all $t$ sufficiently large
\[
|N_t^{\lambda t}| \leq |\hat{N}_{\lfloor t \rfloor}^{\lambda \lfloor t \rfloor}| \leq 
\mathrm{e}^{(\Delta_\lambda + \epsilon)\lfloor t \rfloor}
\]
Then if $\lambda > \lambda_{crit}$ we can take $\epsilon$ sufficiently small so that 
$\Delta_\lambda + \epsilon < 0$ and hence $|N_t^{\lambda t}| < 1$ for $t$ large enough 
thus proving \eqref{supercrit1_upper}. 

If $\lambda < \lambda_{crit}$ then we get
\[
\limsup_{t \to \infty} \frac{1}{t} \log |N_t^{\lambda t}| \leq \Delta_\lambda + \epsilon
\qquad P \text{-a.s.}
\]
and letting $\epsilon \to 0$ yields \eqref{subcrit_upper}.

Finally, if $\lambda > \lambda_{crit}$ then \eqref{supercrit2_upper} follows from Markov inequality and \eqref{log_e_lambda}:
\[
\limsup_{t \to \infty} \frac{1}{t} \log P (|N_t^{\lambda t}| > 0) \leq 
\limsup_{t \to \infty} \frac{1}{t} \log E|N_t^{\lambda t}| = \Delta_{\lambda}.
\]
\end{proof}
\begin{Remark}
Note that the $X \log X$ condition on the offspring distribution was not required so far. It will be essential in the next subsection.
\end{Remark}

\subsection{Additive martingales and applications}
Recall that under the probability $\tilde{P}$ the branching process together with the spine may be described as follows:
\begin{itemize}
\item The process starts with a single spine particle whose path $(\xi_t)_{t \geq 0}$ is distributed like a Brownian motion.
\item At instantaneous rate $\beta_0 \delta_0(\cdot) + \beta$ along its path the spine  particle splits into $A(\cdot)$ particles. If splitting took place at position $x$ then
\begin{equation*}
\tilde{P}(A(x) = n) = \left\{
\begin{array}{rl}
q_n& \text{if } x = 0, \\
p_n& \text{if } x \neq 0. 
\end{array} \right.
\end{equation*}
\item Uniformly at random one of the new particles is selected to continue the spine and thus to stochastically repeat the behaviour of the initial particle starting from $x$.
\item The remaining $A(x) - 1$ particles initiate independent copies of a branching process with branching rate $\beta_0 \delta_0(\cdot) + \beta$ and offsprng distribution $A(\cdot)$ as under $P^x$.
\end{itemize}
We shall now describe a family of martingale changes of measure that will put a certain bias on the motion of the spine particle as well as the birth rate and the offspring distribution along the path of the spine particle. Again, for a detailed discussion the reader is referred to \cite{HH09}.

Let us consider a process of the form 
\begin{align}
\label{m_tilde}
\tilde{M}_t &= \Big[ \prod_{n=1}^{n_t} A_{n} \Big]  \mathrm{e}^{- \hat{\beta}t - \hat{\beta}_0 \tilde{L}_t} \tilde{M}^{(1)}_t \nonumber\\
&= \Big[ \prod_{n=1}^{n_t'} A_{n}' \times \prod_{n=1}^{n_t^0} A_{n}^0 \Big] \mathrm{e}^{- \hat{\beta}t - \hat{\beta}_0 \tilde{L}_t} \tilde{M}^{(1)}_t,
\end{align}
where $(A_n)_{n \geq 1}$, $(A^0_n)_{n \geq 1}$, $(A_n')_{n \geq 1}$, $(n_t)_{t \geq 0}$, $(n^0_t)_{t \geq 0}$ and $(n_t')_{t \geq 0}$ were defined in Subsection 2.1 and $(\tilde{M}^{(1)})_{t \geq 0}$ is a non-negative $\tilde{P}$-martingale of mean $1$ with respect to the filtration $(\mathcal{G}_t)_{t \geq 0}$.  The effect of $\tilde{M}^{(1)}$ if used as a change of measure martingale is to put some drift on $(\xi_t)_{t \geq 0}$. 

For this particular paper we shall take $\tilde{M}^{(1)}$ to be either 
\[
\tilde{M}^{(1)}_t = \mathrm{e}^{\lambda \xi_t - \frac{\lambda^2}{2}t} \text{ , } \qquad t \geq 0
\]
or
\begin{align*}
\tilde{M}^{(1)}_t &= \mathrm{e}^{\lambda |\xi_t| - \lambda \tilde{L}_t - \frac{\lambda^2}{2}t}\\
&= \mathrm{e}^{\lambda \int_0^t sgn(\xi_s) \mathrm{d}\xi_s - \frac{\lambda^2}{2}t}
\text{ , } \qquad t \geq 0.
\end{align*}
The first choice is the classical Girsanov martingle which has the effect of adding constant drift $\lambda$ to $(\xi_t)_{t \geq 0}$. The second choice has the effect of adding instataneous drift $\lambda sgn(\cdot)$ to $(\xi_t)_{t \geq 0}$ so that if $\lambda < 0$ then this is drift of magnitude $|\lambda|$ towards the origin whereas if $\lambda > 0$ then this is drift of magnitude $\lambda$ away from the origin.

Note that we can decompose $(\tilde{M})_{t \geq 0}$ into a product of three martingales:
\[
\tilde{M}_t = \tilde{M}^{(1)}_t \tilde{M}^{(2)}_t \tilde{M}^{(3)}_t
\]
where 
\begin{equation}
\label{m2}
\tilde{M}^{(2)}_t = m^{n_t'} \mathrm{e}^{- \hat{\beta}t} \times 
m_0^{n_t^0} \mathrm{e}^{- \hat{\beta}_0 \tilde{L}_t}
\end{equation}
and
\begin{equation}
\label{m3}
\tilde{M}^{(3)}_t = \prod_{n=1}^{n_t'} \frac{A_{n}'}{m} \times \prod_{n=1}^{n_t^0} \frac{A_{n}^0}{m_0}.
\end{equation}
When used as Radon-Nikodym derivative, $(\tilde{M}^{(2)}_t)_{t \geq 0}$ has the effect of changing the instantaneous jump rate of $(n_t')_{t \geq 0}$ from $\beta$ to $m \beta$ and the jump rate of $(n_t^0)_{t \geq 0}$ from $\beta_0 \delta_0(\cdot)$ to $m_0 \beta_0 \delta_0(\cdot)$. The effect of $(\tilde{M}^{(3)}_t)_{t \geq 0}$ is to change the distribution of random variables $(A_n')_{n \geq 1}$ from $(p_k)_{k \geq 1}$ to $(\frac{k}{m} p_k)_{k \geq 1}$ and the distribution of random variables $(A_n^0)_{n \geq 1}$ from $(q_k)_{k \geq 1}$ to $(\frac{k}{m_0} q_k)_{k \geq 1}$ (while keeping them all independent).
If we now define a new probability measure $\tilde{Q}$ as
\begin{equation}
\frac{\mathrm{d} \tilde{Q}}{\mathrm{d} \tilde{P}} \Big\vert_{\tilde{\mathcal{F}}_t} = 
\tilde{M}_t \text{ , } \qquad t \geq 0 
\end{equation}
then one can check that the effects of $\tilde{M}^{(1)}$, $\tilde{M}^{(2)}$ and $\tilde{M}^{(3)}$ superimpose so that under $\tilde{Q}$ the branching process has the following description.
\begin{itemize}
\item The process starts with a single spine particle whose path $(\xi_t)_{t \geq 0}$ is distributed like a Brownian motion with drift imposed by $\tilde{M}^{(1)}$.
\item At instantaneous rate $m_0 \beta_0 \delta_0(\cdot) + m \beta$ along its path the spine  particle splits into $A(\cdot)$ particles. If splitting took place at position $x$ then
\begin{equation*}
\tilde{Q}(A(x) = n) = \left\{
\begin{array}{rl}
\frac{n}{m_0}q_n& \text{if } x = 0, \\
\frac{n}{m}p_n& \text{if } x \neq 0. 
\end{array} \right.
\end{equation*}
\item Uniformly at random one of the new particles is selected to continue the spine and thus to stochastically repeat the behaviour of the initial particle starting from $x$.
\item The remaining $A(x) - 1$ particles initiate independent unbiased copies of a branching process with branching rate $\beta_0 \delta_0(\cdot) + \beta$ and offspring distribution $A(\cdot)$ as under $P^x$.
\end{itemize}
Suppose now that for all $t \geq 0$ $\tilde{M}^{(1)}_t$ can be represented as 
\[
\tilde{M}^{(1)}_t = \sum_{u \in N_t} M^{(1)}_t(u) \mathbf{1}_{\{ node_t(\xi) = u \}},
\]
where for all $u \in N_t$, $M^{(1)}_t(u)$ is $\mathcal{F}_t$-measurable. E.g. if $\tilde{M}^{(1)}_t = \mathrm{e}^{\lambda |\xi_t| - \lambda \tilde{L}_t - \frac{\lambda^2}{2}t}$ then we get the required representation by taking $M^{(1)}_t(u) = \mathrm{e}^{\lambda |X^u_t| - \lambda L^u_t - \frac{\lambda^2}{2}t}$. If we define 
\[
M_t := \sum_{u \in N_t} M^{(1)}_t(u) \mathrm{e}^{- \hat{\beta}_0 L_t^u - \hat{\beta}t} 
\qquad \text{ , } t \geq 0
\]
it can be checked that $(M_t)_{t \geq 0}$ is a unit-mean $P$-martingale such that 
\[
M_t = \tilde{E} \Big( \tilde{M}_t \big\vert \mathcal{F}_t\Big)
\]
and that if $Q := \tilde{Q} \big\vert_{\mathcal{F}_\infty}$ then 
\[
\frac{\mathrm{d}Q}{\mathrm{d}P} \Big\vert_{\mathcal{F}_t} = M_t \text{ , } \qquad t \geq 0.
\]

Since $(M_t)_{t \geq 0}$ is always a non-negative $P$-martingale, it must converge $P$-almost surely to some non-negative limit $M_\infty$. Particularly interesting are those martingales whose limit is not almost surely $0$.

When investigating whether the limit of $M$ is $P$-almost surely $0$ or not we shall make use of the following result commonly known as the spine decomposition.

\begin{Lemma}[Spine decomposition]
\label{spine_decomp}
Let $(M_t)_{t \geq 0}$ be as above. Then
\begin{equation}
\label{eq_spine_decomp}
\tilde{Q} \Big( M_t \big\vert \tilde{\mathcal{G}}_\infty \Big) = 
spine(t) + \sum_{n=1}^{n_t} (A_n - 1) spine(S_n),
\end{equation}
where 
\[
spine(t) = \tilde{M}^{(1)}_t \mathrm{e}^{- \hat{\beta}t - \hat{\beta}_0 \tilde{L}_t}. 
\]
\end{Lemma}
Note that we have used $\tilde{Q}$ to denote the expectation corresponding to probability measure $\tilde{Q}$ which is a common practice.

We also recall from standard measure theory that if $A \in \mathcal{F}_t$ for some $t \geq 0$ then
\[
Q(A) = \int_A M_t \mathrm{d}P
\]
whereas if $A \in \mathcal{F}_\infty$ then
\begin{equation}
\label{durrett}
Q(A) = \int_A M_\infty \mathrm{d}P + Q \big( A \cap \{ \limsup_{t \to \infty} M_t = \infty\} \big).
\end{equation}
The latter identity can be found in \cite{Durrett}, p.241.
\begin{Proposition}
\label{M_pm}
Let
\begin{equation}
\label{eq_M_pm}
M_t^\pm := \sum_{u \in N_t} \mathrm{e}^{- \hat{\beta}_0 |X^u_t| - \frac{1}{2} \hat{\beta}_0^2 t - \hat{\beta} t} \text{ , } \qquad t \geq 0
\end{equation}
be the $P$-martingale derived through the procedure described above by taking
\begin{equation}
\label{M_one}
\tilde{M}^{(1)}_t = \mathrm{e}^{- \hat{\beta}_0 |\xi_t| + \hat{\beta}_0 \tilde{L}_t - \frac{1}{2} \hat{\beta}_0^2 t} \qquad \text{ , } t \geq 0.
\end{equation}
a) If condition \eqref{xlogx} is satsfied (the $X \log X$ condition on the offspring distribution) then
\[
M^\pm_\infty > 0 \qquad P \text{-a.s.}
\]
b) If condition \eqref{xlogx} is not satsfied then
\[
M^\pm_\infty = 0 \qquad P \text{-a.s.}
\]
\end{Proposition}
Recall that martingale \eqref{M_one} when used as the  Radon-Nikodym derivative has the effect of putting constant drift of magnitude $\hat{\beta}_0$ towards the origin onto $(\xi_t)_{t \geq 0}$. More details on this can be found in \cite{BH14} or \cite{BS02} (``Brownian Motion with alternating drift", pp. 128-129).
\begin{proof}
Let $Q^\pm$ and $\tilde{Q}^\pm$ be the probability measures associated with martingales \eqref{eq_M_pm} and \eqref{M_one} as previously described in this subsection.

A standard argument which relies only on point-recurrence of $(\xi_t)_{t \geq 0}$ under $\tilde{P}$ (see e.g. \cite{JHH09}) tells us that $P(M^\pm_\infty > 0) \in \{0, 1\}$.

So if we could show that $Q^\pm(\limsup_{t \to \infty} M^\pm_t = \infty) = 1$ then by taking $A = \Omega$ in \eqref{durrett} we would get $1 = E M^\pm_\infty + 1$ and hence $M^\pm_\infty = 0 \ P$-almost surely.

On the other hand, if we could show that $Q^\pm(\limsup_{t \to \infty} M^\pm_t = \infty) = 0$ then, again by taking $A = \Omega$ in \eqref{durrett}, we would get $1 = E M^\pm_\infty + 0$ and hence $P(M^\pm_\infty > 0) > 0$, which from the $0-1$ law above is the same as $P(M^\pm_\infty > 0) = 1$.

Thus to prove the proposition it is sufficient to show that if condition \eqref{xlogx} is satisfied then  
$Q^\pm(\limsup_{t \to \infty} M^\pm_t = \infty) = 0$ and if condition \eqref{xlogx} is not satisfied then  
$Q^\pm(\limsup_{t \to \infty} M^\pm_t = \infty) =1$.

Let us observe that because of the effect of martingale $\tilde{M}^{(3)}$ for any choice of $c > 0$ we have
\begin{align*}
\frac{1}{c} \tilde{E}\Big(A_1' \log A_1'\Big) &= \frac{1}{c} m \tilde{Q}^\pm\Big(\log A_1'\Big)\\
&= m \int_0^\infty \tilde{Q}^\pm\Big( \frac{1}{c} \log A_1' \geq t\Big) \mathrm{d}t\\
& = m \sum_{n=0}^\infty \int_n^{n+1} \tilde{Q}^\pm\Big(\log A_1' \geq ct\Big) \mathrm{d}t
\end{align*}
Hence by monotonicity of $\tilde{Q}^\pm\Big(\log A_1' \geq ct\Big)$ we have 
\[
m c \sum_{n=1}^\infty \tilde{Q}^\pm\Big(\log A_1' \geq c n\Big) \leq \tilde{E}\Big(A_1' \log A_1'\Big) 
\leq m c \sum_{n=0}^\infty \tilde{Q}^\pm\Big(\log A_1' \geq c n\Big).
\]
Then since $(A_n')_{n \geq 1}$ are i.i.d. random variables it follows that 
\[
\sum_{n = 1}^\infty \tilde{Q}^\pm\Big(A_n' \geq \mathrm{e}^{cn} \Big) = 
\sum_{n = 1}^\infty \tilde{Q}^\pm\Big(\log A_1' \geq c n\Big) < \infty \iff \tilde{E}\Big(A_1' \log A_1'\Big) = \sum_{n=1}^\infty p_n n \log n < \infty
\]
and then by first and second Borel-Cantelli Lemmas
\begin{equation}
\label{xlogx_dichotomy}
\tilde{Q}^\pm \Big( \big\{ A_n' \geq \mathrm{e}^{cn}\big\} \text{ i.o. }\Big) = \left\{
\begin{array}{rl}
0& \text{if } \sum_{n=1}^\infty p_n n \log n < \infty, \\
1& \text{if } \sum_{n=1}^\infty p_n n \log n = \infty. 
\end{array} \right.
\end{equation}
Identical argument gives 
\begin{equation}
\label{xlogx_dichotomy_0}
\tilde{Q}^\pm \Big( \big\{ A_n^0 \geq \mathrm{e}^{cn}\big\} \text{ i.o. }\Big) = \left\{
\begin{array}{rl}
0& \text{if } \sum_{n=1}^\infty q_n n \log n < \infty, \\
1& \text{if } \sum_{n=1}^\infty q_n n \log n = \infty.
\end{array} \right.
\end{equation}
Let us emphasize that dichotomies \eqref{xlogx_dichotomy} and \eqref{xlogx_dichotomy_0} hold for any choice of $c>0$.

\textit{Proof of a).} Assume now that condition \eqref{xlogx} holds and recall from Lemma \ref{spine_decomp} that 
\begin{align*}
\tilde{Q}^\pm \Big( M^\pm_t \big\vert \tilde{\mathcal{G}}_\infty \Big) &= 
spine(t) + \sum_{n=1}^{n_t'} (A_n' - 1) spine(S_n') + \sum_{n=1}^{n^0_t} (A_n^0 - 1) spine(S^0_n)\\
&\leq 1 + \sum_{n=1}^\infty (A_n' - 1) spine(S_n') + \sum_{n=1}^\infty (A_n^0 - 1) spine(S^0_n),
\end{align*}
where
\[
spine(t) = \mathrm{e}^{- \hat{\beta}_0 |\xi_t| - \frac{1}{2} \hat{\beta}_0^2t - \hat{\beta}t}.
\]
We know that $\tilde{Q}^\pm$-almost surely $\frac{\xi_t}{t} \to 0$, $\frac{\tilde{L}_t}{t} \to \hat{\beta}_0$, $\frac{n_t'}{t} \to m \beta$ and $\frac{n^0_t}{\tilde{L}_t} \to m_0 \beta_0$ as $t \to \infty$.

It follows that $S_n' \sim \frac{1}{m \beta}n$ as $n \to \infty \ \tilde{Q}^\pm$-a.s. so that for any $\delta > 0$
\begin{equation}
\label{spine_bounds}
\mathrm{e}^{-(K+\delta)n} \leq spine(S_n') \leq \mathrm{e}^{-(K-\delta)n} \ \text{eventually } \qquad \tilde{Q}^\pm \text{-a.s.}
\end{equation}
where $K = (\frac{1}{2}\hat{\beta}_0^2 + \hat{\beta})\frac{1}{m \beta} > 0$ which together with the first line of \eqref{xlogx_dichotomy} yields 
\[
\sum_{n=1}^\infty (A_n' - 1) spine(S_n') < \infty \qquad \tilde{Q}^\pm \text{-a.s.}
\]
Similarly $S^0_n \sim \frac{1}{m_0 \beta_0 \hat{\beta}_0}n$ as $n \to \infty$ which 
together with the first line of \eqref{xlogx_dichotomy_0} yields
\[
\sum_{n=1}^\infty (A^0_n - 1) spine(S^0_n) < \infty \qquad \tilde{Q}^\pm \text{-a.s.}
\]
We have thus shown that 
\[
\limsup_{t \to \infty} \tilde{Q}^\pm \Big( M^\pm_t \big\vert \tilde{\mathcal{G}}_\infty \Big) < \infty.
\]
Applying conditional Fatou's Lemma we get
\[
\tilde{Q}^\pm \Big( \liminf_{t \to \infty} M^\pm_t \big\vert \tilde{\mathcal{G}}_\infty\Big) \leq 
\liminf_{t \to \infty} \tilde{Q}^\pm \Big(M^\pm_t \big\vert \tilde{\mathcal{G}}_\infty\Big) \leq 
\limsup_{t \to \infty} \tilde{Q}^\pm \Big(M^\pm_t \big\vert \tilde{\mathcal{G}}_\infty\Big) < \infty
\]
which implies that $\liminf_{t \to \infty} M^\pm_t < \infty$ $\tilde{Q}^\pm$-a.s. and hence also $Q^\pm$-a.s. (because $\{ \liminf_{t \to \infty} M^\pm_t < \infty \} \in \mathcal{F}_\infty$). Then since $\frac{1}{M^\pm}$ is a positive supermartingale under $Q^\pm$ (in fact a true martingale as there is no extinction) it must converge so that 
\[
\limsup_{t \to \infty} M^\pm_t = \liminf_{t \to \infty} M^\pm_t < \infty \qquad \tilde{Q}^\pm\text{-a.s.}
\]
which is sufficient to prove part a) of the proposition.

\textit{Proof of b)} Assume now that $\sum_{n \geq 1} p_n n \log n = \infty$. Then counting only particles born from the spine we get
\[
M^\pm_{S_n'} \geq A_n' spine(S_n')
\]
so that the first inequality in \eqref{spine_bounds} and the second line in \eqref{xlogx_dichotomy} give us that $\tilde{Q}^\pm$ and hence also $Q^\pm$-almost surely
\[
\limsup_{n \to \infty} M^\pm_{S_n'} = \infty.
\]
Therefore we also get
\[
\limsup_{t \to \infty} M^\pm_t = \infty \qquad Q^\pm \text{a.s.}
\]
which proves the sought result.

If $\sum_{n \geq 1} q_n n \log n = \infty$ then we arrive at the same conclusion by replacing $(S_n')_{n \geq 1}$ with $(S^0_n)_{n \geq 1}$ and $(A_n')_{n \geq 1}$ with $(A^0_n)_{n \geq 1}$ in the above argument.
\end{proof}
From Proposition \ref{M_pm} we can now easily derive the required lower bound for Lemma \ref{population_as}.
\begin{Proposition}Lower bound for Lemma \ref{population_as}
\label{lower_pop}$ $

Suppose that condition \eqref{xlogx} on the offspring distribution is satisfied. Then 
\begin{equation}
\label{eq_lower_pop}
\liminf_{t \to \infty} \frac{1}{t} \log |N_t| \geq \frac{1}{2} \hat{\beta}_0^2 + \hat{\beta} \qquad P \text{-a.s.}
\end{equation}
\end{Proposition}
\begin{proof}
\[
|N_t| \mathrm{e}^{- \frac{1}{2} \hat{\beta}_0^2 t - \hat{\beta}t} \geq \sum_{u \in N_t} 
\mathrm{e}^{- \hat{\beta}_0|X^u_t| - \frac{1}{2}\hat{\beta}_0^2t - \hat{\beta}t} = M^\pm_t.
\]
Then
\[
\frac{\log |N_t|}{t} \geq \frac{1}{2} \hat{\beta}_0^2 + \hat{\beta} + \frac{\log M^\pm_t}{t}
\]
and since under condition \eqref{xlogx} $M^\pm_\infty > 0$ $P$-almost surely it follows that
\[
\liminf_{t \to \infty} \frac{1}{t} \log |N_t| \geq \frac{1}{2} \hat{\beta}_0^2 + \hat{\beta} \qquad P \text{-a.s.}
\]
In fact we have an even stronger inequality
\[
\liminf_{t \to \infty} \mathrm{e}^{- \frac{1}{2} \hat{\beta}_0^2 t - \hat{\beta}t} |N_t| \geq 
M^\pm_\infty > 0 \qquad P \text{-a.s.}
\]
\end{proof}
Propositions \ref{lower_pop} and \ref{upper_pop} together prove Lemma \ref{population_as}.

In the rest of this subsection we would like to present some results for a purely homogeneous branching 
process ($\beta_0 = 0$) which we shall make use of in the next section. We begin by stating the following result from \cite{K04} (Theorem 1).
\begin{Proposition}
\label{M_lambda}
Consider a branching process with only homogeneous branching present ($\beta_0 = 0$).
For $\lambda \in \mathbb{R}$ let 
\begin{equation}
\label{eq_M_lambda}
M_t^\lambda := \sum_{u \in N_t} \mathrm{e}^{ \lambda X^u_t - \frac{1}{2} \lambda^2 t - \hat{\beta} t} \text{ , } \qquad t \geq 0
\end{equation}
be the $P$-martingale derived through the procedure described at the begining of this subsection by taking 
\begin{equation}
\label{M_one_lambda}
\tilde{M}^{(1)}_t = \mathrm{e}^{\lambda \xi_t - \frac{1}{2} \lambda^2 t} \qquad \text{ , } t \geq 0.
\end{equation}
a) If $\sum_{n \geq 1} p_n n \log n < \infty$ and $|\lambda| < \sqrt{2 \hat{\beta}}$ then
\[
M^\lambda_\infty > 0 \qquad P \text{-a.s.}
\]
b) If $\sum_{n \geq 1} p_n n \log n < \infty$ and $|\lambda| > \sqrt{2 \hat{\beta}}$ then
\[
M^\lambda_\infty = 0 \qquad P \text{-a.s.}
\]
c) If $\sum_{n \geq 1} p_n n \log n = \infty$ then
\[
M^\lambda_\infty = 0 \qquad P \text{-a.s.}
\]
\end{Proposition}
The proof is essenially the same as that of Proposition \ref{M_pm}. If we define $Q^\lambda$ and $\tilde{Q}^\lambda$ as probability measures associated with martingales \eqref{eq_M_lambda} and \eqref{M_one_lambda} then we would see that under $\tilde{Q}^\lambda$ the $spine(t)$ term would grow exponentially if $|\lambda| > \sqrt{2\hat{\beta}}$ and decay exponentially if $|\lambda| < \sqrt{2\hat{\beta}}$ which together with dichotomy \eqref{xlogx_dichotomy} would lead to the required result.

We shall now make use of Proposition \ref{M_lambda} to get lower bounds on $|N_t^{\lambda t}|$ in purely homogeneous branching systems. 
\begin{Proposition}
\label{pre_lower1}
Consider a branching process with only homogeneous branching present ($\beta_0 = 0$). If $\lambda \in (0, \sqrt{2 \hat{\beta}})$ and $\sum_{n \geq 1} p_n n \log n < \infty$ then
\begin{equation}
\label{eq_pre_lower1}
\liminf_{t \to \infty} \frac{1}{t} \log|N_t^{\lambda t}| \geq \hat{\beta} - \frac{\lambda^2}{2} \qquad 
P \text{-a.s.}
\end{equation}
\end{Proposition}
\begin{proof}
For any choice of $\delta > 0$ such that $\lambda + \delta < \sqrt{2 \hat{\beta}}$ we have the following lower bound on $|N_t^{\lambda t}|$:
\begin{align}
\label{lower_hom}
|N_t^{\lambda t}| &\geq \sum_{u \in N_t} \mathbf{1}_{\{ \lambda t \leq X^u_t \leq (\lambda + 2 \delta)t\}}\nonumber\\
&\geq \sum_{u \in N_t} \mathrm{e}^{(\lambda + \delta)X^u_t - (\lambda + \delta)(\lambda + 2 \delta)t}
\mathbf{1}_{\{ \lambda t \leq X^u_t \leq (\lambda + 2 \delta)t\}}\nonumber\\
&= \mathrm{e}^{\hat{\beta}t - \frac{1}{2}(\lambda + \delta)^2t - \delta(\lambda + \delta)t} \sum_{u \in N_t} \mathrm{e}^{(\lambda + \delta)X^u_t - \frac{1}{2}(\lambda + \delta)^2t - \hat{\beta}t} \mathbf{1}_{\{ \lambda t \leq X^u_t \leq (\lambda + 2 \delta)t\}}
\end{align}
We now claim that as $t \to \infty$
\begin{equation}
\label{mart_indicator1}
\sum_{u \in N_t} \mathrm{e}^{(\lambda + \delta)X^u_t - \frac{1}{2}(\lambda + \delta)^2t - \hat{\beta}t} \mathbf{1}_{\{ \lambda t \leq X^u_t \leq (\lambda + 2 \delta)t\}} \to M^{\lambda + \delta}_\infty 
\qquad P \text{-a.s.,}
\end{equation}
where $M^{\lambda + \delta}$ is the same martingale as in Proposition \ref{M_lambda}. Indeed, 
\begin{align}
\label{mart_indicator2}
&\sum_{u \in N_t} \mathrm{e}^{(\lambda + \delta)X^u_t - \frac{1}{2}(\lambda + \delta)^2t - \hat{\beta}t} \mathbf{1}_{\{ X^u_t > (\lambda + 2 \delta)t\}} \nonumber\\
\leq &\sum_{u \in N_t} \mathrm{e}^{(\lambda + \delta)X^u_t - \frac{1}{2}(\lambda + \delta)^2t - \hat{\beta}t} \mathbf{1}_{\{ X^u_t > (\lambda + 2 \delta)t\}} \mathrm{e}^{\delta X^u_t - \delta(\lambda + 2 \delta)t} \nonumber\\
= &\mathrm{e}^{- \frac{1}{2}\delta^2 t} \sum_{u \in N_t} \mathrm{e}^{(\lambda + 2\delta)X^u_t - \frac{1}{2}(\lambda + 2\delta)^2t - \hat{\beta}t} \mathbf{1}_{\{ X^u_t > (\lambda + 2 \delta)t\}} \nonumber\\
\leq &\mathrm{e}^{- \frac{1}{2}\delta^2 t} M^{\lambda + 2 \delta}_t \to 0 \qquad P \text{-a.s.}
\end{align}
using the fact that $M^{\lambda + 2 \delta}$ converges $P$-almost surely to a finite limit. Similarly, we have
\begin{align}
\label{mart_indicator3}
&\sum_{u \in N_t} \mathrm{e}^{(\lambda + \delta)X^u_t - \frac{1}{2}(\lambda + \delta)^2t - \hat{\beta}t} \mathbf{1}_{\{ X^u_t < \lambda t\}} \nonumber\\
\leq &\sum_{u \in N_t} \mathrm{e}^{(\lambda + \delta)X^u_t - \frac{1}{2}(\lambda + \delta)^2t - \hat{\beta}t} \mathbf{1}_{\{ X^u_t < \lambda t\}} \mathrm{e}^{-\delta X^u_t + \delta \lambda t} \nonumber\\
= &\mathrm{e}^{- \frac{1}{2}\delta^2 t} \sum_{u \in N_t} \mathrm{e}^{\lambda X^u_t - \frac{1}{2}\lambda^2t - \hat{\beta}t} \mathbf{1}_{\{ X^u_t < \lambda t\}} \nonumber\\
\leq &\mathrm{e}^{- \frac{1}{2}\delta^2 t} M^{\lambda }_t \to 0 \qquad P \text{-a.s.}
\end{align}
Thus from \eqref{mart_indicator2} and \eqref{mart_indicator3} it follows that
\begin{align*}
&\sum_{u \in N_t} \mathrm{e}^{(\lambda + \delta)X^u_t - \frac{1}{2}(\lambda + \delta)^2t - \hat{\beta}t} \mathbf{1}_{\{ \lambda t \leq X^u_t \leq (\lambda + 2 \delta)t\}}\\ 
= &M^{\lambda + \delta}_t - \sum_{u \in N_t} \mathrm{e}^{(\lambda + \delta)X^u_t - \frac{1}{2}(\lambda + \delta)^2t - \hat{\beta}t} \mathbf{1}_{\{ X^u_t > (\lambda + 2 \delta)t\}}\\ 
- &\sum_{u \in N_t} \mathrm{e}^{(\lambda + \delta)X^u_t - \frac{1}{2}(\lambda + \delta)^2t - \hat{\beta}t} \mathbf{1}_{\{ X^u_t < \lambda t\}} \to M^{\lambda + \delta}_\infty \qquad P \text{-a.s.}
\end{align*}
proving \eqref{mart_indicator1}. Moreover, from part a) of Proposition \ref{M_lambda} we know that $M^{\lambda + \delta}_\infty > 0 \ P$-almost surely. Hence from \eqref{lower_hom} and \eqref{mart_indicator1} we get
\[
\liminf_{t \to \infty} \frac{1}{t} \log |N_t^{\lambda t}| \geq \hat{\beta} - \frac{1}{2} (\lambda + \delta)^2 - \delta(\lambda + \delta) \qquad{P} \text{-a.s.}
\]
which proves the proposition after letting $\delta \to 0$.
\end{proof}
\begin{Proposition}
\label{pre_lower2}
Consider a branching process with only homogeneous branching present ($\beta_0 = 0$). Let
\begin{equation}
\label{set_ntl}
\tilde{N}_t^{\lambda} := \Big\{ u \in N_{t+1} \ : \ X_s^u > \lambda s \ \forall s \in [t, t+1] \Big\}.
\end{equation}
If $\lambda > \sqrt{2 \hat{\beta}}$ and $\sum_{n \geq 1} p_n n \log n < \infty$ then
\begin{equation}
\label{eq_pre_lower1}
\liminf_{t \to \infty} \frac{1}{t} \log P \big(|\tilde{N}_t^{\lambda}| > 0 \big) \geq \hat{\beta} - \frac{\lambda^2}{2} \qquad P \text{-a.s.}
\end{equation}
In particuar, it is also true that
\begin{equation}
\label{eq_pre_lower1b}
\liminf_{t \to \infty} \frac{1}{t} \log P \big(|N_t^{\lambda t}| > 0 \big) \geq \hat{\beta} - \frac{\lambda^2}{2} \qquad P \text{-a.s.}
\end{equation}
\end{Proposition}
\begin{proof}
For any choice of $\delta > 0$ and $K > 0$ consider the following events:
\[
S^{\lambda, t} := \Big\{ \exists u \in N_{t+1} \ : \ X^u_s > \lambda s \ \forall s \in [t, t+1], \ X^u_s \leq K + (\lambda + 2 \delta)s \ \forall s \in [0, t+1]  \Big\}
\]
and
\[
\tilde{S}^{\lambda, t} := \Big\{ \xi_s > \lambda s \ \forall s \in [t, t+1], \ \xi_s \leq K + (\lambda + 2 \delta)s \ \forall s \in [0, t+1] \Big\}.
\]
One can then see that $S^{\lambda, t} \in \mathcal{F}_{t+1} \subseteq \tilde{\mathcal{F}}_{t+1}$, $\tilde{S}^{\lambda, t} \in \mathcal{G}_{t+1} \subseteq \tilde{F}_{t+1}$ and that\newline 
$\tilde{S}^{\lambda, t} \subseteq S^{\lambda, t} \subseteq \big\{|\tilde{N}_t^{\lambda}| > 0 \big\} \subseteq \big\{|N_t^{\lambda t}| > 0 \big\}$.

We then have the following lower bound on $P(|\tilde{N}_t^{\lambda}| > 0)$:
\[
P \Big( |\tilde{N}_t^{\lambda}| > 0 \Big) \geq P \Big( S^{\lambda, t} \Big)
= E \Big( \mathbf{1}_{S^{\lambda, t}} \frac{M^{\lambda + \delta}_{t+1}}{M^{\lambda + \delta}_{t+1}} \Big)
= Q^{\lambda + \delta} \Big( \mathbf{1}_{S^{\lambda, t}} \frac{1}{M^{\lambda + \delta}_{t+1}}\Big)
= \tilde{Q}^{\lambda + \delta} \Big( \mathbf{1}_{S^{\lambda, t}} \frac{1}{M^{\lambda + \delta}_{t+1}}\Big), 
\]
where $M^{\lambda + \delta}$, $Q^{\lambda + \delta}$ and $\tilde{Q}^{\lambda + \delta}$ are the same as in Proposition \ref{M_lambda}. Then
\[
\tilde{Q}^{\lambda + \delta} \Big( \mathbf{1}_{S^{\lambda, t}} \frac{1}{M^{\lambda + \delta}_{t+1}}\Big) 
\geq \tilde{Q}^{\lambda + \delta} \Big( \mathbf{1}_{\tilde{S}^{\lambda, t}} \frac{1}{M^{\lambda + \delta}_{t+1}}\Big) 
\geq \tilde{Q}^{\lambda + \delta} \Big( \mathbf{1}_{\tilde{S}^{\lambda, t}} \frac{1}{\tilde{Q}^{\lambda + \delta} \big( M^{\lambda + \delta}_{t+1} \vert \tilde{\mathcal{G}}_\infty\big)}\Big) 
\]
using conditional Jensen inequality and the pull-through property of conditional expectation in the last inequality. 
We now recall that 
\[
\tilde{Q}^{\lambda + \delta} \Big( M^{\lambda + \delta}_{t+1} \big\vert \tilde{\mathcal{G}}_\infty\Big)
= spine(t+1) + \sum_{n=1}^{n_{t+1}} (A_n - 1)spine(S_n),
\]
where 
\[
spine(t) = \mathrm{e}^{(\lambda + \delta)\xi_t - \frac{1}{2}(\lambda + \delta)^2t - \hat{\beta}t}.
\]
On the event $\tilde{S}^{\lambda, t}$ we have that for all $s \in [0,t+1]$
\begin{align*}
spine(s) \leq &\exp \Big\{(\lambda + \delta) \big( K +(\lambda + 2\delta)s \big) - \frac{1}{2}(\lambda + \delta)^2s - \hat{\beta}s \Big\}\\ 
= &\mathrm{e}^{K(\lambda + \delta)} \exp \Big\{\big( \frac{1}{2}(\lambda + \delta)^2 + \delta(\lambda + \delta) - \hat{\beta}\big)s \Big\}\\
\leq &C_\delta \exp \Big\{\big( \frac{1}{2}(\lambda + \delta)^2 + \delta(\lambda + \delta) - \hat{\beta}\big)t \Big\},
\end{align*}
where $C_\delta$ is some positive constant. Also from dichotomy \eqref{xlogx_dichotomy} we know that $A_n < \mathrm{e}^{\delta n}$ eventually $\tilde{Q}^{\lambda + \delta}$-almost surely. Thus
\[
\sum_{n=1}^{n_{t+1}} A_n \leq \sum_{n =1}^{n_{t+1}} \mathrm{e}^{\delta n} + Y \leq C_\delta' \mathrm{e}^{\delta n_{t+1}} + Y,
\]
where $Y = \sum_{n=1}^\infty A_n \mathbf{1}_{\{A_n > \mathrm{e}^{\delta n}\}}$ is a $\tilde{Q}^{\lambda + \delta}$-almost surely finite random variable independent of $n_{t+1}$ and $(\xi_s)_{0 \leq s \leq t+1}$ and $C_\delta'$ is some positive constant. Thus 
\[
\tilde{Q}^{\lambda + \delta} \Big( M^{\lambda + \delta}_{t+1} \big\vert \tilde{\mathcal{G}}_\infty\Big) 
\leq C_\delta \mathrm{e}^{( \frac{1}{2}(\lambda + \delta)^2 + \delta(\lambda + \delta) - \hat{\beta})t} \big( 1 + Y + C_\delta' \mathrm{e}^{\delta n_{t+1}} \big).
\]
Then using the fact that $\frac{1}{a + b} \geq \frac{1}{2ab}$ whenever $a$, $b \geq 1$ we get
\begin{align*}
&\tilde{Q}^{\lambda + \delta} \Big( \mathbf{1}_{\tilde{S}^{\lambda, t}} \frac{1}{\tilde{Q}^{\lambda + \delta} \big( M^{\lambda + \delta}_{t+1} \vert \tilde{\mathcal{G}}_\infty\big)}\Big)\\ 
\geq &\frac{1}{C_\delta}\mathrm{e}^{\big(\hat{\beta} - \frac{1}{2}(\lambda + \delta)^2 - \delta(\lambda + \delta)\big)t} \ \tilde{Q}^{\lambda + \delta} \Big( \mathbf{1}_{\tilde{S}^{\lambda, t}} 
\frac{1}{1 + Y + C_\delta' \mathrm{e}^{\delta n_{t+1}}}\Big)\\
\geq &\frac{1}{2 C_\delta C_\delta'} \mathrm{e}^{\big(\hat{\beta} - \frac{1}{2}(\lambda + \delta)^2 - \delta(\lambda + \delta)\big)t} \tilde{Q}^{\lambda + \delta} \Big( \tilde{S}^{\lambda, t} \Big) \tilde{Q}^{\lambda + \delta} \Big( \mathrm{e}^{- \delta n_{t+1}} \Big) \tilde{Q}^{\lambda + \delta} \Big( \frac{1}{1 + Y} \Big).
\end{align*}
We note that $\tilde{Q}^{\lambda + \delta}\big(\frac{1}{1 + Y}\big) > 0$ since $Y$ is $\tilde{Q}^{\lambda + \delta}$-almost surely finite, $\tilde{Q}^{\lambda + \delta}(\mathrm{e}^{- \delta n_{t+1}}) = \mathrm{e}^{m \beta (t+1) (\mathrm{e}^{- \delta} - 1)}$ since $(n_t)_{t \geq 0}$ is a $\tilde{Q}^{\lambda + \delta}$-Poisson process with rate $m \beta$ and $\tilde{Q}^{\lambda + \delta} ( \tilde{S}^{\lambda, t}) \to C_{K, \delta}$ for some positive constant $C_{K, \delta}$ since $(\xi_t)_{t \geq 0}$ is a Brownian motion with drift $\lambda + \delta$ under $\tilde{Q}^{\lambda + \delta}$.
Therefore
\begin{align*}
\liminf_{t \to \infty} \frac{1}{t} \log P \Big( |\tilde{N}_t^{\lambda}| > 0 \Big) \geq 
&\liminf_{t \to \infty} \frac{1}{t} \tilde{Q}^{\lambda + \delta} \Big( \mathbf{1}_{\tilde{S}^{\lambda, t}} \frac{1}{\tilde{Q}^{\lambda + \delta} \big( M^{\lambda + \delta}_{t+1} \vert \tilde{\mathcal{G}}_\infty\big)}\Big)\\
\geq &\hat{\beta} - \frac{1}{2}(\lambda + \delta)^2 - \delta(\lambda + \delta) - m \beta(1 - \mathrm{e}^{-\delta})
\end{align*}
which proves the required result after letting $\delta \to 0$.
\end{proof}
\begin{Remark}
\label{remark_lower_bound}
Note that Proposition \ref{pre_lower1} and Proposition \ref{pre_lower2} (equation \eqref{eq_pre_lower1b}) already provide sufficient lower bounds for Lemma \ref{pop_lambda_as} (equations \eqref{subcrit} and \eqref{supercrit2} respectively) in the case $\lambda \geq \hat{\beta}_0$ since a spatially-homogeneous branching process can be embedded in a process with homogeneous and catalytic branching both present by simply not counting any particles born due to catalytic branching.
\end{Remark}

\section{Remaining proofs}
In this subsection we shall complete the proof of Lemma \ref{pop_lambda_as} by establishing lower bounds for equations \eqref{subcrit} and \eqref{supercrit2}. We shall then finish off the paper with the proof of Corollary \ref{cor_rightmost}.

\subsection{Heuristic argument}
Here we discuss the idea behind the proof in a non-rigorous way in order to help the reader understand the formal argument given in the next subsection.

Our task is to find the optimal way for a particle to reach level $\lambda t$ at some large time $t$. 

In the case of spatially-homogeneous branching ($\beta_0 = 0$) the birth rate along the path of a particle is independent of the path and so the optimal way would simply be to travel at speed $\lambda$ all the time (there are of course finer results available, but they are irrelevant to this  discussion).

However, in the presence of the catalyst at the origin travelling at speed $\lambda$ all the time might be disadvantageous as it will discard any contribution from the catalyst. Thus one might think that a better strategy for a  particle would first be to stay near the origin for some positive proportion of time in order to give birth to more particles at an accelerated rate (due to both homogeneous and catalytic branching potential) and then for the remaining time let its children travel at whatever speed necessary in order to reach the required level.

The argument goes as follows. For a large time $t$ we let 
\begin{equation*}
q := \left\{
\begin{array}{rl}
|N_t^{\lambda t}|& \text{if } \lambda < \lambda_{crit}, \\
P(|N_t^{\lambda t}| > 0)& \text{if } \lambda > \lambda_{crit} 
\end{array} \right.
\end{equation*}
and we want a lower bound on $q$.

We fix a number $p \in [0,1]$. As we know from Lemma \ref{population_as} at time $pt$ there are
\[ 
|N_{pt}| \approx \exp \Big\{ \big(\frac{\hat{\beta}_0^2}{2} + \hat{\beta} \big)pt \Big\}
\]
partices in the system and about a half of them lie in the upper-half plane. Next we ignore any catalytic branching that takes place between times $pt$ and $t$ 
by assuming that every particle $u \in N_{pt}^0$ starts an independent spatially-homogeneous branching process from the position $X^u_{pt} > 0$.

If we let 
\begin{equation*}
q(u) := \left\{
\begin{array}{rl}
|N_T^{\frac{\lambda}{1-p}T}(u)|& \text{if } \frac{\lambda}{1-p} < \sqrt{2 \hat{\beta}},\\
P(|N_T^{\frac{\lambda}{1-p}T}(u)| > 0)& \text{if } \frac{\lambda}{1-p} > \sqrt{2 \hat{\beta}},
\end{array} \right.
\end{equation*}
where $N_T^{\frac{\lambda}{1-p}T}(u)$ is the set of particles which lie above level $\frac{\lambda}{1-p}T$ 
at time $T$ of the spatially-homogeneous process initiated by $u$ in the time-space frame of this process 
and where $T = (1-p)t$ then from Propositions \ref{pre_lower1} and \ref{pre_lower2} we know that
\begin{align*}
q(u) &\gtrsim \exp \Big\{ \hat{\beta}T - \frac{1}{2} \Big( \frac{\lambda}{1-p}\Big)^2 T \Big\}\\
&= \exp \Big\{ \hat{\beta}(1-p)t - \frac{\lambda^2}{2(1-p)} t \Big\}.
\end{align*}
Then since every particle in $N_T^{\frac{\lambda}{1-p}T}(u)$ for every $u \in N_{pt}^0$ also 
belongs to $N_t^{\lambda t}$, we can estimate
\begin{align}
\label{q_estimate}
q &\gtrsim \sum_{u \in N_{pt}^0} q(u)\nonumber \\
&\approx |N_{pt}^0| \exp \Big\{ \hat{\beta}(1-p)t - \frac{\lambda^2}{2(1-p)} t \Big\} \nonumber \\
&\approx \exp \Big\{ \big( \hat{\beta} + \frac{\hat{\beta}_0^2}{2}p - \frac{\lambda^2}{2(1-p)} 
\big) t \Big\}.
\end{align}
The value of $p$ which maximises this expression is 
\begin{equation}
\label{p_star}
p^{\ast} := \left\{
\begin{array}{rl}
1 - \frac{\lambda}{\hat{\beta}_0}& \text{if } \lambda \leq \hat{\beta}_0,\\
0& \text{if } \lambda \geq \hat{\beta}_0.
\end{array} \right.
\end{equation}
Substituting this value of $p$ into \eqref{q_estimate} we get
\begin{align*}
q &\gtrsim \left\{
\begin{array}{rl}
\exp \big\{ \big(\hat{\beta} + \frac{\hat{\beta}_0^2}{2} - \hat{\beta}_0 \lambda \big) 
t\big\} & \text{if } \lambda \leq \hat{\beta}_0,\\
\exp \big\{ \big(\hat{\beta} - \frac{\lambda^2}{2} \big) t\big\}& \text{if } \lambda \geq \hat{\beta}_0
\end{array} \right.\\
&= \exp \{ \Delta_\lambda t \}
\end{align*}
which gives the lower bound on $q$ that we want. 

Note that if $\lambda$ is too large ($\lambda \geq \hat{\beta}_0$) then $p^\ast = 0$ and so the best 
strategy for a particle to reach level $\lambda t$ at time $t$ would indeed be to travel at speed $\lambda$ all the time being driven by homogeneous branching potential with negligible contribution from catalytic branching. This is consistent with Remark \ref{remark_lower_bound} made earlier.

\subsection{Lower bounds for \eqref{subcrit} and \eqref{supercrit2}}
Before we present the main body of the proof let us give a couple of preliminary results. The first one is a very crude estimate of the number of particles which approximately lie in the upper-half plane at a  large time $t$.
\begin{Proposition}
\label{positive_population}
Assume that condition \eqref{xlogx} on the offspring distribution is satisfied. Then $P$-almost surely for any $\delta > 0$ there exists a finite time $T_\delta$ such that for all $t \geq T_\delta$
\begin{equation}
\label{eq_positive_pop}
|N^{- \delta t}_t| \geq \exp \Big\{\big(\hat{\beta} + \frac{\hat{\beta}_0^2}{2} - c_\delta\big)t \Big\},
\end{equation}
where $c_\delta$ is some positive constant with the property that $c_\delta \to 0$ as $\delta \to 0$.
\end{Proposition}
\begin{proof}
Let us observe that $\Delta_\delta \to \hat{\beta} + \frac{\hat{\beta}_0^2}{2}$ as $\delta \to 0$. So we may write $\Delta_\delta = \hat{\beta} + \frac{\hat{\beta}_0^2}{2} - c_\delta'$ for some $c_\delta' > 0$ such that $c_\delta' \to 0$ as $\delta \to 0$.

From Lemma \ref{population_as} (or Proposition \ref{lower_pop}) we know that $P$-almost surely for any $\delta > 0$ there exists a finite time $T_\delta'$ such that for all $t \geq T_\delta'$
\begin{equation}
\label{eq_positive_pop1}
|N_t| \geq \exp \Big\{\big(\hat{\beta} + \frac{\hat{\beta}_0^2}{2} - \frac{1}{4}c_\delta' \big)t \Big\}.
\end{equation}
We also know from Proposition \ref{upper_pop_lambda} (equation \eqref{subcrit_upper}) that $P$-almost surely for any $\delta > 0$ there exists a finite time $T_\delta''$ such that for all $t \geq T_\delta''$
\[
|N^{\delta t}_t| \leq \exp \Big\{\big(\Delta_\delta + \frac{1}{2} c_\delta' \big)t \Big\} = 
\exp \Big\{ \big(\hat{\beta} + \frac{\hat{\beta}_0^2}{2} - \frac{1}{2}c_\delta' \big)t \Big\}.
\]
Thus by symmetry it is also true that $P$-almost surely for any $\delta > 0$ there exists a finite time $T_\delta'''$ such that for all $t \geq T_\delta'''$
\begin{equation}
\label{eq_positive_pop2}
|N_t| - |N^{- \delta t}_t| \leq \exp \Big\{ \big(\hat{\beta} + \frac{\hat{\beta}_0^2}{2} - \frac{1}{2}c_\delta' \big)t \Big\}.
\end{equation}
Subtracting \eqref{eq_positive_pop2} from \eqref{eq_positive_pop1} yields the result.
\end{proof}
The next result is basically a version of Chebyshev's inequality.
\begin{Proposition}
\label{chebyshev}
Let $N$ be a random variable supported on $\mathbb{N}$ and $(S_k)_{k \geq 1}$ a sequence of events independent of each other conditional on $N$. If for some $r \in (0, 1)$ it is true that $\mathbb{P}(S_k|N) \geq r$ $\mathbb{P}$-a.s. for all $k \geq 1$ then
\[
\mathbb{P} \Big( \sum_{k = 1}^N \mathbf{1}_{S_k} \leq \frac{r}{2}N \ \Big\vert N \Big)
\leq \frac{4}{rN} \qquad \mathbb{P}-a.s.
\]
\end{Proposition}
Sharper inequalities are of course available but are excessive to us.
\begin{proof}
\begin{align*}
\mathbb{P} \Big( \sum_{k = 1}^N \mathbf{1}_{S_k} \leq \frac{r}{2}N \ \Big\vert N \Big) 
\leq &\mathbb{P} \Big( \sum_{k = 1}^N \big( \mathbf{1}_{S_k} - \mathbb{P}(S_k|N) \big) \leq - \frac{1}{2}
\sum_{k = 1}^N \mathbb{P}(S_k|N) \ \Big\vert N \Big) \\
\leq &\mathbb{P} \Big( \big\vert \sum_{k = 1}^N \big( \mathbf{1}_{S_k} - \mathbb{P}(S_k|N) \big) \big\vert \geq \frac{1}{2} \sum_{k = 1}^N \mathbb{P}(S_k|N) \ \Big\vert N \Big)\\
\leq &\frac{\sum_{k=1}^N var(\mathbf{1}_{S_k}|N)}{\frac{1}{4} \big(\sum_{k = 1}^N \mathbb{P}(S_k|N) \big)^2} \leq \frac{4}{rN} \qquad \mathbb{P}-a.s.
\end{align*}
using conditional Markov inequality and the fact that $var(\mathbf{1}_{S_k}|N) = \mathbb{P}(S_k|N) - \mathbb{P}(S_k|N)^2 \leq \mathbb{P}(S_k|N)$.
\end{proof}
\begin{Proposition}[Lower bounds for Lemma \ref{pop_lambda_as}]
\label{lower_pop_lambda} Suppose that condition \eqref{xlogx} on the offspring distribution is satisfied. 

\noindent If $\lambda < \lambda_{crit}$ ($\Delta_\lambda > 0$) then 
\begin{equation}
\label{subcrit_lower}
\liminf_{t \to \infty} \frac{1}{t} \log |N_t^{\lambda t}| \geq \Delta_\lambda 
\qquad P \text{-a.s.}
\end{equation}
If $\lambda > \lambda_{crit}$ ($\Delta_\lambda < 0$) then 
\begin{equation}
\label{supercrit_lower}
\liminf_{t \to \infty} \frac{1}{t} \log P \big(|N_t^{\lambda t}| > 0 \big) \geq \Delta_\lambda.
\end{equation}
\end{Proposition}
\begin{proof}
We let $p = p^\ast$ be the same as in \eqref{p_star}. We note that $p = 0$ if and only if $\lambda \geq \hat{\beta_0}$ while if $\lambda \geq \hat{\beta_0}$ then $\Delta_\lambda = \hat{\beta} - \frac{1}{2}\lambda^2$ so that \eqref{subcrit_lower} and \eqref{supercrit_lower} follow from Proposition \ref{pre_lower1} and Proposition \ref{pre_lower2} (equation \eqref{eq_pre_lower1b}) by simply not counting any particles born due to catalytic branching. Thus for the rest of the proof we shall assume that $\lambda < \hat{\beta}_0$ so that $p = 1 - \frac{\lambda}{\hat{\beta}_0} > 0$ and $\Delta_\lambda = \hat{\beta} + \frac{\hat{\beta}_0^2}{2} - \hat{\beta}_0 \lambda$. 

We then choose some $\delta > 0$ and define
\begin{equation}
\label{lambda_hat}
\hat{\lambda} := \frac{\lambda + \delta}{\lambda} \hat{\beta}_0.
\end{equation}
We also define 
\[
f(\delta) := \big( \hat{\beta} - \frac{\hat{\lambda}^2}{2} - \delta \big),
\]
\[
g(\delta) := \big( \hat{\beta} + \frac{\hat{\beta}_0^2}{2} - c_\delta \big),
\]
where $c_\delta$ is the same as in Proposition \ref{positive_population}. We let $h(\delta)$ be such that 
\begin{align*}
(1-p)f(\delta) + pg(\delta) = &\big( \hat{\beta} - \frac{\hat{\lambda}^2}{2} - \delta \big)(1-p) + 
\big( \hat{\beta} + \frac{\hat{\beta}_0^2}{2} - c_\delta \big)p\\
= &\hat{\beta} + \frac{\hat{\beta}_0^2}{2} \big( 1 - \frac{\lambda}{\hat{\beta}_0} \big) - c_\delta p
- \frac{1}{2} \Big( \frac{\lambda + \delta}{\lambda} \hat{\beta}_0 \Big)^2 \frac{\lambda}{\hat{\beta}_0} - \delta(1-p)\\
= &\Delta_\lambda - h(\delta).
\end{align*}
Note that $h(\delta)>0$ and $h(\delta) \to 0$ as $\delta \to 0$.

For $t>0$ we define events 
\begin{equation}
\label{eq_A}
A_t := \Big\{ |N_{pt}^{- \delta pt}| \geq \mathrm{e}^{g(\delta) pt} \Big\}.
\end{equation}
From Proposition \ref{positive_population} we know that $P(A_n \ ev.) = 1$ so that in particular $P(A_t) \to 1$ as $t \to \infty$. 

Finally, for every particle $u \in N_{pt}$ in the time-space frame of the subtree initiated by $u$ we define $N_s(u)$ to be the set of particles in this subtree at time $s$ with $Y_s^v$ positions of particles $v \in N_s(u)$ at time $s$. Moreover, by analogy with \eqref{set_ntx} and \eqref{set_ntl} we define
\[
N_s^x(u) := \Big\{ v \in N_s(u) \ : \ Y_s^v > x\Big\}
\]
and
\[
\tilde{N}_s^l(u) := \Big\{ v \in N_{s+1}(u) \ : \ Y_r^v > lr \ \forall r \in [s, s+1]\Big\}.
\]

\underline{Proof of the lower bound for \eqref{subcrit}}

Assume that $\lambda < \lambda_{crit}$ so that $\Delta_\lambda > 0$. There are two cases to consider which require slightly different treatment.

\textbf{Case 1:} $\hat{\beta}_0 < \sqrt{2 \hat{\beta}}$.

We choose $\delta > 0$ to be sufficiently small that $\hat{\lambda} < \sqrt{2 \hat{\beta}}$ and for $n \geq 1$ consider events
\[
B_n := \Big\{ \sum_{u \in N_{pn}^{- \delta pn}} \mathbf{1}_{ B_n(u)} < \frac{1}{4} |N_{pn}^{- \delta pn}| \Big\},
\]
where for every $u \in N_{pn}^{- \delta pn}$
\[
B_n(u) = \Big\{ |N_s^{\hat{\lambda}s}(u)| \geq \mathrm{e}^{ f(\delta) s } \ \text{ for all } s \in \big[(1-p)n, (1-p)n + 1\big] \Big\}.
\]
We know that conditional on $\mathcal{F}_{pn}$ events $B_n(u)$ are independent (since all the subtrees initiated by particles $u \in N_{pn}$ are independent copies of the original branching process started from positions $X^u_{pn}$). 

Moreover, if we ignore all the catalytic branching taking place in the subtrees initiated by particles $u \in N_{pn}^{- \delta p n}$ then we can get from Proposition \ref{pre_lower1} that there exists some deterministic $n_0$ such that for all $n \geq n_0$, $P(B_n(u) | \mathcal{F}_{pn}) \geq \frac{1}{2}$. Hence by Proposition \ref{chebyshev} $P(B_n | \mathcal{F}_{pn}) \leq \frac{8}{|N_{pn}^{- \delta pn}|}$ $P$-a.s. for all $n \geq n_0$. Then for all $n \geq n_0$ we get that
\[
P\Big(A_n \cap B_n\Big) = E \Big( E \big( \mathbf{1}_{A_n} \mathbf{1}_{B_n} \big\vert \mathcal{F}_{pn}\big) \Big) \leq E \Big( \mathbf{1}_{A_n} \frac{8}{|N_{pn}^{- \delta pn}|} \Big) \leq 
8 \mathrm{e}^{- g(\delta) pn}
\]
which decays exponentially fast in $n$ (for $\delta$ sufficiently small). Therefore $P(A_n \cap B_n \text{ i.o.}) = 0$. Then since $P(A_n \text{ ev.}) = 1$ it follows that $P(A_n \cap B_n^c \text{ ev.}) = 1$. So  $P$-a.s. for all $n$ large enough 
\[
\sum_{u \in N_{pn}^{- \delta pn}} \mathbf{1}_{ B_n(u)} \geq \frac{1}{4} |N_{pn}^{- \delta pn}| 
\geq \frac{1}{4} \mathrm{e}^{g(\delta)pn}.
\]
Then noting that for all $t \in [n, n+1]$
\[
|N_t^{\lambda t}| \geq \sum_{u \in N_{pn}^{- \delta pn}} \mathbf{1}_{ B_n(u)} \mathrm{e}^{f(\delta)(1-p)n}
\]
we get that $P$-a.s. for all $t$ sufficiently large
\[
|N_t^{\lambda t}| \geq K \mathrm{e}^{g(\delta)pt + f(\delta)(1-p)t},
\]
where $K$ is some positive constant. Hence 
\[
\liminf_{t \to \infty} \frac{1}{t} \log |N_t^{\lambda t}| \geq \Delta_\lambda - h(\delta) \qquad P \text{-a.s.,}
\]
which yields the required value after letting $\delta \to 0$.

\textbf{Case 2:} $\hat{\beta}_0 \geq \sqrt{2 \hat{\beta}}$ (so that $\hat{\lambda} > \sqrt{2\hat{\beta}}$).

For $n \geq 1$ we consider events
\[
C_n := \Big\{ \sum_{u \in N_{pn}^{- \delta pn}} \mathbf{1}_{\big\{ |\tilde{N}_{(1-p)n}^{\hat{\lambda}}(u)| > 0\big\}} < \frac{1}{2} \mathrm{e}^{f(\delta)(1-p)n} |N_{pn}^{- \delta pn}|\Big\}.
\]
We know that conditional on $\mathcal{F}_{pn}$ events $\{|\tilde{N}^{\hat{\lambda}}_{(1-p)n}(u)|>0\}$ are independent. Moreover, if we ignore all the catalytic branching taking place in the subtrees initiated by particles $u \in N_{pn}^{- \delta p n}$ then we can get from Proposition \ref{pre_lower2} that there exists some deterministic $n_0$ such that for all $n \geq n_0$, $P(|\tilde{N}^{\hat{\lambda}}_{(1-p)n}(u)|>0 | \mathcal{F}_{pn}) \geq \mathrm{e}^{f(\delta)(1-p)n}$. Hence by Proposition \ref{chebyshev} for all $n \geq n_0$
\[
P\big(C_n | \mathcal{F}_{pn}\big) \leq \frac{4}{|N_{pn}^{- \delta pn}|} \mathrm{e}^{-f(\delta)(1-p)n} \qquad P \text{-a.s.}
\]
Then for all $n \geq n_0$ we get that
\begin{align*}
P\Big(A_n \cap C_n\Big) = &E \Big( E \big( \mathbf{1}_{A_n} \mathbf{1}_{C_n} \big\vert \mathcal{F}_{pn}\big) \Big)
\leq E \Big( \mathbf{1}_{A_n} \frac{4}{|N_{pn}^{- \delta pn}|} 
\mathrm{e}^{-f(\delta)(1-p)n}\Big)\\
\leq &4\mathrm{e}^{-g(\delta)p -f(\delta)(1-p)}
= 4 \mathrm{e}^{ -(\Delta_\lambda - h(\delta) ) n},
\end{align*}
which decays exponentially fast in $n$ (for $\delta$ chosen sufficiently small). Therefore $P(A_n \cap C_n \text{ i.o.}) = 0$. Then since $P(A_n \text{ ev.}) = 1$ it follows that $P(A_n \cap C_n^c \text{ ev.}) = 1$. So  $P$-a.s. for all $n$ large enough 
\[
\sum_{u \in N_{pn}^{- \delta pn}} \mathbf{1}_{\big\{ |\tilde{N}_{(1-p)n}^{\hat{\lambda}}(u)| > 0\big\}} \geq \frac{1}{2} \mathrm{e}^{f(\delta)(1-p)n} |N_{pn}^{- \delta pn}|
\geq 4 \mathrm{e}^{(\Delta_\lambda - h(\delta) )n}.
\]
Then noting that for all $t \in [n, n+1]$
\[
|N_t^{\lambda t}| \geq \sum_{u \in N_{pn}^{- \delta pn}} \mathbf{1}_{ \big\{ |\tilde{N}_{(1-p)n}^{\hat{\lambda}}(u)| > 0\big\}} 
\]
we get that 
\[
\liminf_{t \to \infty} \frac{1}{t} \log |N_t^{\lambda t}| \geq \Delta_\lambda - h(\delta) \qquad P 
\text{-a.s.}
\]
which yields the required result after letting $\delta \to 0$.

\underline{Proof of the lower bound for \eqref{supercrit2}}

Assume that $\lambda > \lambda_{crit}$ so that $\Delta_\lambda < 0$. Then necessarily 
$\hat{\lambda} > \hat{\beta}_0 > \lambda > \lambda_{crit} \geq \sqrt{2 \hat{\beta}}$. 

We note that
\[
P \Big( |N_t^{\lambda t}| > 0\Big) \geq 
P \Big( \bigcup_{u \in N_{pt}^{- \delta pt}} \big\{ |N^{\hat{\lambda}t}_{(1-p)t}(u)| > 0 \big\} \ , \ |N_{pt}^{- \delta pt}| \geq \mathrm{e}^{g(\delta)pt} \Big).
\]
We know that conditional on $\mathcal{F}_{pt}$ events $\{|N^{\hat{\lambda}t}_{(1-p)t}(u)|>0\}$ are independent. Moreover, if we ignore all the catalytic branching taking place in the subtrees initiated by particles $u \in N_{pt}^{- \delta p t}$ then we can get from Proposition \ref{pre_lower2} (equation \eqref{eq_pre_lower1b}) that there exists some deterministic $t_0$ such that for all $t \geq t_0$, $P \big(|N^{\hat{\lambda}t}_{(1-p)t}(u)|>0 \ | \mathcal{F}_{pt} \big) \geq \mathrm{e}^{f(\delta)(1-p)t}$. Hence 
\begin{align*}
&P \Big( \bigcup_{u \in N_{pt}^{- \delta pt}} \big\{ |N^{\hat{\lambda}t}_{(1-p)t}(u)| > 0 \big\}, \ |N_{pt}^{- \delta pt}| \geq \mathrm{e}^{g(\delta)pt} \Big)\\
= &E \Big( \mathbf{1}_{\big\{|N_{pt}^{- \delta pt}| \geq \mathrm{e}^{g(\delta)pt}\big\}} \Big[ 1 - \prod_{u \in N_{pt}^{- \delta pt}} \big( 1 - 
P(|N^{\hat{\lambda}t}_{(1-p)t}(u)|>0 | \mathcal{F}_{pt})\big)\Big]\Big)\\
\geq &P \Big( |N_{pt}^{- \delta pt}| \geq \mathrm{e}^{g(\delta)pt} \Big) \Big[ 1 - \Big( 1 - \mathrm{e}^{f(\delta)(1-p)t} \Big)^{\mathrm{e}^{g(\delta )pt}} \Big]\\
\geq &P \Big( |N_{pt}^{- \delta pt}| \geq \mathrm{e}^{g(\delta)pt} \Big) \mathrm{e}^{(\Delta_\lambda - h(\delta))t}
\end{align*}
for all $t$ large enough and where in the last inequality we have used the fact that for any $a \in (0,1)$ and $b > 1$ such that $ab < 1$ it is true that $(1-a)^b \leq 1 - ab$. Then noting that as $t \to \infty$
\[
P \Big( |N_{pt}^{- \delta pt}| \geq \mathrm{e}^{g(\delta)pt} \Big) \to 1 \qquad P \text{-a.s.}
\]
we get that 
\[
\liminf_{t \to \infty} \frac{1}{t} \log P \big(|N_t^{\lambda t}| > 0 \big) \geq \Delta_\lambda - h(\delta)
\]
which yields the required result after letting $\delta \to 0$.

\end{proof}

\subsection{Proof of Corollary \ref{cor_rightmost} and final remarks}
\begin{proof}[Proof of Corollary \ref{cor_rightmost}]
Assume that condition \eqref{xlogx} is satisfied. Then for any $\lambda < \lambda_{crit}$, as we know from \eqref{subcrit}, $|N_t^{\lambda t}| > 0$ for all $t$ large enough $P$-almost surely. So for all $t$ large enough $R_t \geq \lambda t$ $P$-almost surely. Thus $\liminf_{t \to \infty} \frac{R_t}{t} \geq \lambda$ $P$-a.s. Then letting $\lambda \to \lambda_{crit}$ gives
\[
\liminf_{t \to \infty} \frac{R_t}{t} \geq \lambda_{crit} \qquad P \text{-a.s.}
\]
Similarly, if $\lambda > \lambda_{crit}$ then by \eqref{supercrit1} $|N_t^{\lambda t}| = 0$ for all $t$ large enough $P$-almost surely. So for all $t$ large enough $R_t \leq \lambda t$ $P$-almost surely. Thus $\limsup_{t \to \infty} \frac{R_t}{t} \geq \lambda$ $P$-a.s. Then letting $\lambda \to \lambda_{crit}$ gives
\[
\limsup_{t \to \infty} \frac{R_t}{t} \leq \lambda_{crit} \qquad P \text{-a.s.}
\]
which completes the proof.
\end{proof}

\end{document}